\documentclass[10pt,reqno]{amsart}

\usepackage{amsmath}
\usepackage{amssymb}

\usepackage[bitstream-charter]{mathdesign}
\usepackage[T1]{fontenc}

\newtheorem{theorem}{Theorem}
\newtheorem{lemma}[theorem]{Lemma}
\newtheorem{corollary}[theorem]{Corollary}

\newtheorem{rem}[theorem]{Remark}

\newcommand{\bbB}{{\mathbb B}}
\newcommand{\bbI}{{\mathbb I}}
\newcommand{\bbK}{{\mathbb K}}
\newcommand{\R}{{\mathbb R}}
\newcommand{\bbU}{{\mathbb U}}

\newcommand{\bfC}{\mathbf{C}}
\newcommand{\bfL}{\mathbf{L}}
\newcommand{\bfW}{\mathbf{W}}

\newcommand{\bfb}{\mathbf{b}}

\newcommand{\bfn}{\mathbf{n}}
\newcommand{\bfu}{\mathbf{u}}
\newcommand{\bfx}{\mathbf{x}}
\newcommand{\bfy}{\mathbf{y}}

\newcommand{\bfphi}{\mathbf{\phi}}

\newcommand{\cB}{\mathcal{B}}
\newcommand{\cF}{\mathcal{F}}
\newcommand{\cM}{\mathcal{M}}
\newcommand{\cS}{\mathcal{S}}
\newcommand{\cT}{\mathcal{T}}

\newcommand{\rmd}{\mathrm{d}}

\newcommand{\bfzero}{\mathbf{0}}
\newcommand{\br}{\hbox to 0.7pt{}}
\newcommand{\Ls}{\bfL_{\sigma}}
\newcommand{\Cns}{\bfC_{0,\sigma}}
\newcommand{\Ws}{\bfW_{\sigma}}

\newcommand{\bfbr}{\bfb^{\bfx_0}_{\rm r}}
\newcommand{\bfbp}{\bfb^{\bfx_0}_{\rm p}}
\newcommand{\bfur}{\bfu^{\bfx_0}_{\rm r}}
\newcommand{\bfup}{\bfu^{\bfx_0}_{\rm p}}

\newcommand{\lclosed}{[\hspace{0.5pt}}
\newcommand{\rclosed}{\hspace{0.5pt}]}

\newcommand{\qand}{\quad \text{ and } \quad}

\DeclareMathOperator*{\divg}{div}
\DeclareMathOperator*{\curl}{curl}
\DeclareMathOperator*{\esssup}{ess\,sup}

\begin{document}
\baselineskip=18pt

\title[New Regularity Criteria for Weak Solutions to the MHD Equations]{New Regularity Criteria for Weak Solutions to the MHD Equations in Terms of an Associated Pressure}

\author{Ji\v{r}\'{\i} Neustupa \and Minsuk Yang}

\address{J. Neustupa: 
Czech Academy of Sciences, Institute of Mathematics, Praha, Czech Republic}
\email{neustupa@math.cas.cz}

\address{M. Yang: 
Yonsei University, Department of Mathematics, Seoul, Republic of Korea}
\email{m.yang@yonsei.ac.kr}

\begin{abstract}
We prove that if $0<T_0<T\leq\infty$, $(\mathbf{u},\mathbf{b},p)$ 
is a suitable weak solution of the MHD equations in $\mathbb{R}^3\times(0,T)$ 
and either $\mathcal{F}_{\gamma}(p_-)\in L^{\infty}(0,T_0;\,L^{3/2}(\mathbb{R}^3))$ 
or $\mathcal{F}_{\gamma}((|\mathbf{u}|^2+|\mathbf{b}|^2+2p)_+) 
\in L^{\infty}(0,T_0;\, L^{3/2}(\mathbb{R}^3))$ for some $\gamma>0$, 
where $\mathcal{F}_{\gamma}(s)=s\, [\ln (1+s)]^{1+\gamma}$ 
and the subscripts ``$-$'' and ``$+$'' denote the negative 
and the nonnegative part, respectively, then the solution 
$(\mathbf{u},\mathbf{b},p)$ has no singular points 
in $\mathbb{R}^3\times(0,T_0]$. The results are also valid, 
as a special case, for the Navier--Stokes equations.
\end{abstract}

\maketitle

\section{Introduction and formulation of the main results} 
\label{S1}

\subsection{The system of MHD equations.} 
\label{SS1.1} 

The motion of a viscous incompressible electrically conductive 
fluid in $\R^3$ in the time interval $(0,T)$ (where $0<T\leq\infty$), 
at the absence of an external specific body force and an external 
magnetic induction, is described by the system of 
\emph{magneto-hydro-dynamical equations} 
(which is abbreviated to \emph{MHD equations})
\begin{align}
\rho\, \partial_t\bfu+\rho\br\bfu\cdot\nabla\bfu-\mu\, \curl\,
\bfb\times\bfb\ &=\ -\nabla p+\rho\br\nu\br\Delta\bfu, \label{1.1} \\ 
\partial_t\bfb-\curl\, (\bfu\times\bfb)\ &=\ -\br\xi\,
\curl\, (\curl \br\bfb), \label{1.2} \\ 
\divg\bfu\ =\ \divg\bfb\ &=\ 0. \label{1.3}
\end{align}
The system is completed by the initial conditions
\begin{equation}
\bfu\br\bigl|_{t=0}=\bfu_0 \qand \bfb\br\bigl|_{t=0}=\bfb_0.
\label{1.4}
\end{equation}
The unknowns are the velocity field $\bfu$ of the fluid, the
magnetic field $\bfb$ and the pressure $p$. The coefficients
$\rho$, $\mu$, $\nu$ and $\xi$, which are all supposed to be
positive constants, represent the density of the fluid, the
magnetic permeability, the kinematic viscosity and the magnetic
diffusivity, respectively. We may further assume, without loss of
generality, that $\rho=1$ and $\mu=1$. Then the equations
(\ref{1.1}) and (\ref{1.2}) can also be written in the form
\begin{align*}
\partial_t\bfu+\bfu\cdot\nabla\bfu-\bfb\cdot
\nabla\bfb+\nabla\bigl(p+{\textstyle\frac{1}{2}}\br |\bfb|^2\bigr)
\ &=\ \nu\Delta\bfu, \tag{1} \\
\partial_t\bfb+\bfu\cdot\nabla\bfb-\bfb\cdot\nabla\bfu\ &=\
\xi\Delta\bfb \tag{2}.
\end{align*}

\subsection{Notation.} 
\label{SS1.2}

We denote vector functions and spaces of vector functions 
by boldface letters. We denote by $\Cns(\R^3)$ the linear space 
of all infinitely differentiable divergence--free
vector functions in $\R^3$ with a compact support and 
by $\Ls^2(\R^3)$ the closure of $\Cns(\R^3)$ in $\bfL^2(\R^3)$.
We define $\Ws^{1,2}(\R^3):=\bfW^{1,2}(\R^3)\cap\Ls^2(\R^3)$.
Finally, we will denote by $c$ a general constant, which may change its value in the same proof.

\subsection{Weak solutions, suitable weak solutions, and associated pressure.} 
\label{SS1.3} 

Given $\bfu_0,\, \bfb_0 \in \Ls^2(\R^3)$, a pair 
$(\bfu,\bfb)\in \bigl[ L^{\infty}(0,T;\, \Ls^2(\R^3)) 
\cap L^2(0,T$; $\Ws^{1,2}(\R^3)) \bigr]^2$ is said
to be a \emph{weak solution} to the system
(\ref{1.1})--(\ref{1.3}) with the initial conditions
(\ref{1.4}) if the integral identities
\begin{align*}
\int_0^T\!\int_{\R^3} \bigl[
-\bfu\cdot\partial_t\bfphi & +\bfu\cdot\nabla\bfu\cdot\bfphi-\bfb\cdot
\nabla\bfb\cdot\bfphi+\nu\br\nabla\bfu:\nabla\bfphi \bigr]\;
\rmd\bfx\, \rmd t 
\ =\ \int_{\R^3} \bfu_0\cdot\bfphi(\, .\, ,0)\; \rmd\bfx \\
\int_0^T\!\int_{\R^3} \bigl[ -\bfb\cdot\partial_t\bfphi & +
\bfu\cdot\nabla\bfb\cdot\bfphi-\bfb\cdot\nabla\bfu\cdot\bfphi+
\xi\br\nabla\bfb:\nabla\bfphi\bigr]\; \rmd\bfx\, \rmd t 
\ =\ \int_{\R^3} \bfb_0\cdot\bfphi(\, .\, ,0)\; \rmd\bfx
\end{align*}
hold for all $\bfphi\in C^{\infty}_0\bigl( [0,T);\,
\Cns(\R^3)\bigr)$.

A distribution $p$ in $Q_T:=\R^3\times(0,T)$ is said to be an {\it
associated pressure} if $\bfu$, $\bfb$ and $p$ satisfy equations
(\ref{1.1})--(\ref{1.3}) in the sense of distributions in $Q_T$.

If $(\bfu,\bfb)$ is a weak solution, an associated pressure
$p$ is a locally integrable function in $Q_T$ such that the
product $p\br\bfu$ is also locally integrable in $Q_T$, and
$\bfu,\, \bfb,\, p$ satisfy the so called \emph{localized energy
inequality}
\begin{align*}
&\int_{Q_T}2\br\bigl(\nu\, |\nabla\bfu|^2+\xi\,
|\nabla\bfb|^2\bigr)\, \psi\; \rmd\bfx\, \rmd t \\
&\leq\ \int_{Q_T}
\bigl[\br |\bfu|^2\, (\partial_t\psi+\nu\Delta\psi)+
|\bfb|^2\, (\partial_t\psi+\xi\Delta\psi) \\
&\qquad \qquad +\, \bigl(|\bfu|^2+|\bfb|^2+2p\bigr)\,
(\bfu\cdot\nabla\psi)-2\mu\br (\bfu\cdot\bfb)\,
(\bfb\cdot\nabla\psi)\br\bigr]\; \rmd\bfx\, \rmd t 
\end{align*}
for every nonnegative infinitely differentiable scalar function
$\psi$ compactly supported in $Q_T$, then we call $(\bfu,\bfb,p)$
a \emph{suitable weak solution} to the system
(\ref{1.1})--(\ref{1.3}).

Note that the existence of a weak solution can be proven by the
same method as for the Navier--Stokes equations. The existence of
an associated pressure and its smoothness are studied in the
paper \cite{NeYa} by J.~Neustupa and M.~Yang. The sketch of the
construction of a suitable weak solution (which is also analogous
to the Navier--Stokes equations) can be found in the paper
\cite{HeXi2} by Ch.~He and Z.~Xin.

\subsection{Regular points and singular points.} 
\label{SS1.4} 

Throughout the paper we assume that $(\bfu,\bfb,p)$ is a suitable weak
solution to the MHD initial-value problem
(\ref{1.1})--(\ref{1.4}) in $Q_T$. A space-time point
$(\bfx_0,t_0)\in Q_T$ is said to be a \emph{regular point} of
the solution $(\bfu,\bfb,p)$ if there exists a neighborhood
$U(\bfx_0,t_0)\subset Q_T$ of this point such that both $\bfu$
and $\bfb$ are essentially bounded in $U(\bfx_0,t_0)$. Other
points of $Q_T$ are called {singular points.} It follows from
the paper \cite{MaNiSh} by A.~Mahalov, B.~Nicolaenko and
T.~Shilkin that if $(\bfx_0,t_0)$ is a regular point of the
solution $(\bfu,\bfb,p)$ then there exists a neighborhood
$U(\bfx_0,t_0)$, such that $\bfu$ and $\bfb$ together with all
their spatial derivatives (of all orders) are
H\"older--continuous in $U(\bfx_0,t_0)$. Moreover, J.~Neustupa
and M.~Yang \cite{NeYa} showed that the neighborhood
$U(\bfx_0,t_0)$ can be chosen so that $\partial_t\bfu$ and $p$ 
together with all their spatial derivatives (of all orders) 
are essentially bounded in $U(\bfx_0,t_0)$.

Suppose that $M\subset(0,T)$ and let us denote by $\cS_M$ the
set of all singular points of the solution $(\bfu,\bfb,p)$ in
$\R^3\times M$ and by $\cS(t_0)$ (for $t_0\in(0,T)$) the set of
all points $\bfx\in\R^3$ such that $(\bfx,t_0)\in\cS_{(0,T)}$.
(It should be noted that the question whether $\cS_{(0,T)}$ is
nonempty is open.) Obviously, the set $\cS_{(0,T)}$ is closed
in $\R^3\times(0,T)$ and the set $\cS(t_0)$ is closed in
$\R^3$. Ch.~He and Z.~Xin \cite{HeXi2} derived a series of
criteria for regularity of the solution $(\bfu,\bfb,p)$ at a
given point $(\bfx_0,t_0)\in Q_T$, from which one can deduce
that the $1$-dimensional Hausdorff measure of $\cS_{(0,T)}$ is
zero.

\subsection{The choice of the pressure.} \label{SS1.5} 
As the pressure $p$ can be modified by an additive function of $t$, 
we fix $p$ so that
\begin{align}
p(\bfx,t) + \frac{1}{2}\, |\bfb(\bfx,t)|^2 \
=\ \frac{1}{4\pi}\int_{\R^3}\frac{1}{|\bfx-\bfy|}\
\divg\divg\bigl[\bfu(\bfy,t)\otimes\bfu(\bfy,t)-\bfb(\bfy,t)\otimes
\bfb(\bfy,t)\bigr]\; \rmd\bfy. \label{1.5}
\end{align}
This formula comes from the equation
\begin{displaymath}
\Delta\bigl(p+{\textstyle\frac{1}{2}}\br|\bfb|^2\bigr)\ =\
\divg\divg[\bfu\otimes\bfu-\bfb\otimes\bfb]\ \equiv\
\partial_i\partial_j(u_iu_j-b_ib_j),
\end{displaymath}
which we obtain if we apply the operator ${\rm div}$ to
equation (\ref{1.1}). We explain in subsection \ref{SS2.1} that
formula (\ref{1.5}) has a sense for all $(\bfx,t)\in
Q_T\smallsetminus\cS_{(0,T)}$. The pressure given by
(\ref{1.5}) satisfies $\nabla p\in L^r(\delta,T;\,
\bfL^s(\R^3))$ for all $0<\delta<T$, $1<r<2$,
and $1<s<3/2$ satisfying $2/r+3/s=4$ (see Theorem 3 in
\cite{NeYa}). The functions $\bfu,\, \bfb$ are supposed to have
been modified on a set of measure zero so that both $\bfu$ and
$\bfb$ are weakly continuous from $(0,T)$ to $\Ls^2(\R^3)$.

\subsection{The results of this paper.} \label{SS1.6}
To neatly formulate
our main theorem and write its proof, we introduce the following
function. Let $\gamma$ be a positive parameter and define 
for $s\geq 0$ the function 
\begin{equation}
\cF_{\gamma}(s)\ :=\ s\ \lclosed\ln{}(1+s)\rclosed^{1+\gamma}.
\label{1.6}
\end{equation}
We note that $\cF_{\gamma}$ is increasing and strictly convex 
on $[0,\infty)$.

Here is our main theorem:

\begin{theorem} \label{T1}
Let $\bfu_0,\, \bfb_0\in\Ws^{1,2}(\R^3)$ and let $(\bfu,\bfb,p)$
be a suitable weak solution of the MHD initial-value problem
(\ref{1.1})--(\ref{1.4}) in $Q_T$. Let $0<T_0<T$ and suppose that
there exists $\gamma>0$ so that at least one of the conditions

\begin{list}{}
{\setlength{\topsep 1mm}
\setlength{\itemsep 1mm}
\setlength{\leftmargin 24pt}
\setlength{\rightmargin 0pt}
\setlength{\labelwidth 16pt}}

\item[a) ]
$\cF_{\gamma}(p_-)\in L^{\infty}(0,T_0;\, L^{3/2}(\R^3))$,

\item[b) ]
$\cF_{\gamma}(\cB_+)\in L^{\infty}(0,T_0;\, L^{3/2}(\R^3))$, \
where \ $\cB:= \frac{1}{2}\br|\bfu|^2+\frac{1}{2}\br|\bfb|^2+p$

\end{list}
holds. Then the set $\cS_{(0,T_0]}$ of singular points of the
solution $(\bfu,\bfb,p)$ in $\R^3\times(0,T_0]$ is empty.
Consequently, the functions $\bfu$ and $\bfb$ are H\"older
continuous in $\R^3\times(0,T_0]$.
\end{theorem}

\noindent
Note that the subscripts ``$-$'' and ``$+$'' denote the
negative and nonnegative part, respectively. As the negative part is
taken ``positively'', e.g.~$p$ satisfies $p=p_+-p_-$.
The function $\cB$ is the so called \emph{Bernoulli pressure.}
Obviously, condition a), respectively b), is satisfied if there
exists $q>\frac{3}{2}$ such that $p_-$, respectively $\cB_+$, is
in $L^{\infty}(0,T_0;\, L^q(\R^3))$.
It follows from Theorem \ref{T1} that if $\cF_{\gamma}(p_-)$ or
$\cF_{\gamma}(\cB_+)$ lies in $L^{\infty}(0,T;\,
L^{3/2}(\R^3))$ then $\cS_{(0,T)}=\emptyset$.

\subsection{Comparison with previous results.} \label{SS1.7}
In many papers, various authors have formulated sufficient
conditions for regularity of a weak solution to the Navier--Stokes
equations in terms of an associated pressure. In this context, we
quote the papers \cite{BeGa}, \cite{BoCoPa}, \cite{CaFaZh},
\cite{ChaeLee}, \cite{ChZha}, \cite{FaOz}, \cite{FaJiNi},
\cite{KaLee1}, \cite{KaLee2}, \cite{NeNe}, \cite{Ne1},
\cite{SeSve}, \cite{Str2} and \cite{Su}. A typical idea used in
most of the papers is to multiply the Navier--Stokes equation by
the function $\bfu\, |\bfu|^{\alpha}$ with an appropriate
$\alpha>0$ and then to integrate over the spatial domain where
the equations are considered. Then the integral of 
$(\bfu\cdot\nabla\bfu)\cdot\bfu\, |\bfu|^{\alpha}$ is equal to
zero whenever $\bfu$ satisfies the no--slip boundary condition. 
This method, however, fails in the case of the MHD
equations. The reason is that the momentum equation (\ref{1.1})
contains, in addition to $\bfu\cdot\nabla\bfu$, also the nonlinear
term $\bfb\cdot\nabla\bfb$, and this term multiplied by $\bfu\,
|\bfu|^{\alpha}$ does not lead to zero. This cannot be
compensated by equation (\ref{1.2}) multiplied e.g.~by $\bfb\,
|\bfb|^{\alpha}$ or anything else. Of papers, based on another
method than is the sketched idea, we quote \cite{NeNe}, \cite{Ne1}
and \cite{SeSve}. In the last cited paper, G.~Seregin and
V.~\v{S}ver\'ak consider a suitable weak solution $\bfu$, $p$ to
the Navier--Stokes equations in $\R^3\times(0,\infty)$. The
authors say that a scalar function
$g:\R^3\times(0,\infty)\to[0,\infty)$ satisfies condition (C) if
to any $t_0>0$ there exists $R_0>0$ such that
\begin{displaymath}
A(t_0)\ :=\ \sup_{\bfx_0\in\R^3}\ \sup_{t_0-R_0^2\leq t\leq t_0}\
\int_{|\bfx-\bfx_0|<R_0}\frac{g(\bfx,t)}{|\bfx-\bfx_0|}\;
\rmd\bfx\ <\ \infty
\end{displaymath}
and for each fixed $\bfx_0\in\R^3$ and each fixed $R\in(0,R_0]$,
the function
\begin{displaymath}
t\ \longmapsto\ \int_{|\bfx-\bfx_0|<R}\frac{g(\bfx,t)}
{|\bfx-\bfx_0|}\; \rmd\bfx
\end{displaymath}
is left continuous at $t_0$. The main result of
\cite{SeSve} says that if there exists $g$ satisfying
condition (C) so that the normalized pressure
\begin{displaymath}
p(\bfx,t)\ :=\ \frac{1}{4\pi}\int_{\R^3}\frac{1}{|\bfx-\bfy|}\
[\partial_iu_j\, \partial_ju_i](\bfy,t)\; \rmd\bfy
\end{displaymath}
satisfies
\begin{equation}
|\bfu(\bfx,t)|^2+2p(\bfx,t)\ \leq g(\bfx,t) \qquad \mbox{for all}\
\bfx\in\R^3,\ 0<t<\infty \label{1.7}
\end{equation}
or
\begin{equation}
p(\bfx,t)\ \geq\ -g(\bfx,t) \qquad \mbox{for all}\ \bfx\in\R^3,\
0<t<\infty \label{1.8}
\end{equation}
then $\bfu$ is H\"older-continuous in $\R^3\times(0,\infty)$,
i.e.~$\bfu$ is regular. The method has been extended by K.~Kang
and J.~Lee to the MHD equations (\ref{1.1})--(\ref{1.3}). (See
Theorem 1.3 in \cite{KaLee3}.) The first sufficient condition
for regularity of a suitable weak solution $\bfu$, $\bfb$, $p$
of the MHD system (\ref{1.1})--(\ref{1.3}) formulated in
\cite{KaLee3} coincides with (\ref{1.8}). The second condition
is similar to (\ref{1.7}): it requires
\begin{equation}
|\bfu(\bfx,t)|^2+|\bfb(\bfx,t)|^2+2p(\bfx,t)\ \leq g(\bfx,t)
\qquad \mbox{for all}\ \bfx\in\R^3,\ 0<t<\infty. \label{1.9}
\end{equation}
Our conditions a) and b), used in Theorem \ref{T1}, are weaker
than conditions (\ref{1.8}) and (\ref{1.9}) from paper
\cite{KaLee3}, respectively. Thus, our Theorem \ref{T1}
generalizes Theorem 1.3 from paper \cite{KaLee3}. Obviously, if
one considers $\bfb\equiv\bfzero$ then our Theorem \ref{T1} also
represents a generalization of the results from \cite{SeSve}.

\subsection{Two auxiliary results.} \label{SS1.8}
We finish this section by giving two lemmas, 
which will later clarify the reasons for the use of
function $\cF_{\gamma}$ in conditions a) and b) of Theorem
\ref{T1}.

\begin{lemma} 
\label{L1.1}
Let $\gamma > 0$ and $f$ a nonnegative measurable function in $\R^3$.
If $\cF_{\gamma}(f)\in L^{3/2}(\R^3)$, then there exists 
$R_\gamma \in (0,1]$ such that
\begin{equation}
\sup_{0<r<R_\gamma} \sup_{\bfx_0\in\R^3} 
\left(\frac{(\ln r^{-1})^{1+\gamma}}{r} 
\int_{B_r(\bfx_0)} f(\bfx)\; \rmd\bfx\right)\ < \infty. 
\label{1.10}
\end{equation}
\end{lemma}

\begin{proof}
Let $M_1:= \{\bfx \in \R^3;\ f(\bfx)\leq 1\}$ and
$M_2:=\{\bfx \in \R^3;\ f(\bfx) > 1\}$.
Define 
\begin{align*}
J_1(\bfx_0,r)\ &:=\ \frac{(\ln r^{-1})^{1+\gamma}}{r} 
\int_{B_r(\bfx_0) \cap M_1} f(\bfx)\; \rmd\bfx, \\
J_2(\bfx_0,r)\ &:=\ \frac{(\ln r^{-1})^{1+\gamma}}{r} 
\int_{B_r(\bfx_0) \cap M_2} f(\bfx)\; \rmd\bfx.
\end{align*}
We clearly have 
\[
J_1(\bfx_0,r)\ \leq\ \frac{(\ln r^{-1})^{1+\gamma}}{r} 
\int_{B_r(\bfx_0)} \rmd\bfx\ =\
\frac{4\pi}{3}\ r^2(\ln r^{-1})^{1+\gamma},
\]
and $\lim_{r\to 0+}\ r^2(\ln r^{-1})^{1+\gamma} = 0$ after a simple computation.
Thus, there exists $\rho_\gamma \in (0,1]$ such that 
\begin{equation}
\label{1.11}
\sup_{0<r<\rho_\gamma} \sup_{\bfx_0\in\R^3} J_1(\bfx_0,r)\ <\ 1.
\end{equation}
Now, we focus on estimating the main term $J_2(\bfx_0,r)$.
We put $g(\bfx):=f(\bfx)\, \chi_{M_2}(\bfx)$ and $\chi_{M_2}$ to denote the
characteristic function of the set $M_2$ so that  
\begin{equation}
J_2(\bfx_0,r)\ =\ \frac{(\ln r^{-1})^{1+\gamma}}{r} 
\int_{B_r(\bfx_0)} g(\bfx)\; \rmd\bfx.
\label{1.12}
\end{equation}
We define for $s\geq 0$ 
\begin{equation}
\Phi_{\gamma}(s)\ :=\ \cF_{\gamma}^{3/2}(s)\ =\ s^{\frac{3}{2}}\
[\ln{}(1+s)]^{\frac{3(1+\gamma)}{2}}.
\label{1.13}
\end{equation}
Notice that $\Phi_{\gamma}$ is increasing and strictly convex on $[0,\infty)$
since $\cF_{\gamma}$ has the same properties. 
By Jensen's inequality we have 
\[
\Phi_{\gamma}\biggl( \frac{3}{4\pi r^3}\int_{B_r(\bfx_0)} g(\bfx)\; \rmd\bfx \biggr)\ 
\leq\ \frac{3}{4\pi r^3} \int_{B_r(\bfx_0)} \Phi_{\gamma}(g)(\bfx)\; \rmd\bfx.
\]
Since $\cF_{\gamma}(f)\in L^{3/2}(\R^3)$, if we denote 
\begin{equation}
c_1\ :=\ \frac{3}{4\pi} \int_{\R^3} \Phi_{\gamma}(g)(\bfx)\; \rmd\bfx\ <\ \infty,
\label{1.14}
\end{equation}
then 
\[
\Phi_{\gamma}\biggl( \frac{3}{4\pi r^3}\int_{B_r(\bfx_0)}g(\bfx)\; 
\rmd\bfx \biggr)\ \leq\  r^{-3} c_1.
\]
Since $\Phi_{\gamma}$ is bijective and its inverse function 
$\Phi_{\gamma}^{-1}$ is increasing, we have
\begin{equation}
\int_{B_r(\bfx_0)} g(\bfx)\; \rmd\bfx\
\le\ \frac{4\pi r^3}{3}\ \Phi_{\gamma}^{-1}(r^{-3} c_1).
\label{1.15}
\end{equation}
If we denote 
\begin{equation}
A_{\gamma}(r)\ :=\ \Phi_{\gamma}^{-1}(r^{-3}c_1),
\label{1.16}
\end{equation}
then combining \eqref{1.12}, \eqref{1.15}, and \eqref{1.16} we obtain 
\begin{align*}
J_2(\bfx_0,r)\ 
&\le\ \frac{4\pi}{3}\ (\ln r^{-1})^{1+\gamma}\ r^2 A_{\gamma}(r) \\
&=\ \frac{4\pi}{3}\ \Bigl( \frac{\ln r^{-1}}
{\ln{}(1+A_{\gamma}(r))} \Bigr)^{1+\gamma}\ r^2 A_{\gamma}(r)\, 
[\ln{}(1+A_{\gamma}(r))]^{1+\gamma}.
\end{align*}
By the definition \eqref{1.13} we have 
\[
\cF_{\gamma}^{3/2}(A_{\gamma}(r))\ =\ \Phi_{\gamma}(A(r))\ =\ r^{-3} c_1.
\]
Hence $r^2\, \cF_{\gamma}(A_{\gamma}) = c_1^{2/3}$, which is the same as 
\[
r^2 A_{\gamma}(r)\,
[\ln{}(1+A_{\gamma}(r))]^{1+\gamma}\ =\ c_1^{2/3}
\]
due to the definition \eqref{1.6}.
Thus, we get 
\begin{equation}
J_2(\bfx_0,r)\ 
\le\ \frac{4\pi}{3}\ c_1^{2/3}\ \Bigl( \frac{\ln r^{-1}}
{\ln{}(1+A_{\gamma}(r))} \Bigr)^{1+\gamma}.
\label{1.17}
\end{equation}
Notice that 
\[
\frac{\ln r^{-1}}{\ln{}(1+A_{\gamma}(r))}\ <\ 1
\]
is equivalent to $r^{-1}-1 < A_{\gamma}(r)$, which is also equivalent to 
\[
r^3\, \Phi_{\gamma}(r^{-1}-1)\ <\ r^3\, \Phi_{\gamma}(A_{\gamma}(r))= c_1
\]
where we have used the fact that $\Phi_{\gamma}$ is increasing and \eqref{1.16}.
By definition we have
\[
r^3\ \Phi_{\gamma}(r^{-1}-1)=\bigl(r^2\, \cF_{\gamma}(r^{-1}-1)\bigr)^{3/2}.
\]
After some simple computations, we see that 
\[
\lim_{r\to 0+}\ r^2\, \cF_{\gamma}(r^{-1}-1)\ =\ \lim_{r\to 0+}\
r(1-r)\ (\ln r^{-1})^{1+\gamma}\ =\ 0.
\]
Thus, there exists a positive number $R_\gamma < \rho_\gamma$ such that 
\[
\sup_{0<r<R_\gamma} r^3\, \Phi_{\gamma}(r^{-1}-1)\ <\ c_1.
\]
Equivalently, we have 
\begin{equation}
\sup_{0<r<R_\gamma} \frac{\ln r^{-1}}{\ln{}(1+A_{\gamma}(r))}\ <\ 1.
\label{1.18}
\end{equation}
Therefore combining \eqref{1.17} and \eqref{1.18} we get 
\begin{equation}
\sup_{0<r<R_\gamma} \sup_{\bfx_0\in\R^3} J_2(\bfx_0,r)\ <\ \frac{4\pi}{3}\ c_1^{2/3}.
\label{1.19}
\end{equation}
From \eqref{1.11} and \eqref{1.19} we get the desired result \eqref{1.10}.
\end{proof}

\begin{rem}
\label{R1}
An upper bound of \eqref{1.10} can be taken as the sum of 
the bounds in \eqref{1.11} and \eqref{1.19}, which depends only 
on the $L^{3/2}(\R^3)$ norm of $\cF_{\gamma}(f)$.
\end{rem}

\begin{lemma} \label{L1.2}
Let $\gamma > 0$ and $f$ a nonnegative measurable function in $\R^3$.
If there exists $R \in (0,1]$ such that
\begin{equation}
\sup_{0<r<R} \sup_{\bfx_0\in\R^3} \left(\frac{(\ln r^{-1})^{1+\gamma}}{r} 
\int_{B_r(\bfx_0)} f(\bfx)\; \rmd\bfx\right)\ < \infty, \label{1.20}
\end{equation}
then 
\[
\lim_{r\to 0+}\ \sup_{\bfx_0\br\in\br\R^3}\ \int_{B_r(\bfx_0)}
\frac{f(\bfx)}{|\bfx-\bfx_0|}\; \rmd\bfx = 0.
\]
\end{lemma}

\begin{proof}
Define the measure $\rmd\mu=f(\bfx)\, \rmd\bfx$ on $\R^3$.
Applying Fubini's theorem we get 
\[
\int_{B_r(\bfx_0)} \frac{f(\bfx)}{|\bfx-\bfx_0|}\; \rmd\bfx\
=\ \int_{B_r(\bfx_0)} \frac{\rmd\mu}{|\bfx-\bfx_0|}\ 
=\ \int_0^{\infty} \mu \bigl\{ \bfx\in B_r(\bfx_0);\
|\bfx-\bfx_0|^{-1}>\xi \bigr\}\; \rmd\xi.
\]
Replacing $\xi=\zeta^{-1}$ and changing variables we have 
\[
\ \int_0^{\infty} \mu \bigl\{ \bfx\in B_r(\bfx_0);\
|\bfx-\bfx_0|^{-1}>\xi \bigr\}\; \rmd\xi\ 
=\ \int_0^{\infty} \mu \bigl\{ \bfx\in B_r(\bfx_0);\
|\bfx-\bfx_0|<\zeta \bigr\}\; \frac{\rmd\zeta}{\zeta^2}.
\]
Splitting the last integral into two parts we define 
\begin{align*}
D_1(\bfx_0,r)\ &:=\ \int_0^r \mu \bigl\{ \bfx\in B_r(\bfx_0);\ \, 
|\bfx-\bfx_0|<\zeta \bigr\}\; \frac{\rmd\zeta}{\zeta^2}, \\
D_2(\bfx_0,r)\ &:=\ \int_r^{\infty} \mu \bigl\{ \bfx\in B_r(\bfx_0);\ \, 
|\bfx-\bfx_0|<\zeta \bigr\}\; \frac{\rmd\zeta}{\zeta^2}.
\end{align*}
Let $c$ denote the supremum in \eqref{1.20}.
Then we use the condition \eqref{1.20} to obtain 
\[
D_1(\bfx_0,r)\ =\ \int_0^r \mu\bigl(B_{\zeta}(\bfx_0)\bigr)\; 
\frac{\rmd\zeta}{\zeta^2}\
=\ \int_0^r \biggl(\int_{B_{\zeta}(\bfx_0)} f(\bfx)\;
\rmd\bfx\biggr)\; \frac{\rmd\zeta}{\zeta^2}\
\le\ c \int_0^r \frac{\zeta}{(\ln\zeta^{-1})^{1+\gamma}}\;
\frac{\rmd\zeta}{\zeta^2}.
\]
Changing variables we have 
\[
\int_0^r \frac{\zeta}{(\ln\zeta^{-1})^{1+\gamma}}\; \frac{\rmd\zeta}{\zeta^2}\
=\ \int_0^r\frac{1}{\zeta\, (-\ln\zeta)^{1+\gamma}}\; \rmd\zeta\ 
=\ \int_{-\ln r}^{\infty}\frac{\rmd\eta}{\eta^{1+\gamma}}\ 
=\ \frac{1}{\gamma\, (\ln r^{-1})^{\gamma}}.
\]
Thus, we get 
\begin{equation}
\lim_{r\to 0+}\ \sup_{\bfx_0\br\in\br\R^3}\ D_1(\bfx_0,r)\ 
\le\ c \lim_{r\to 0+}\ \frac{1}{\gamma\, (\ln r^{-1})^{\gamma}}\ =\ 0.
\label{1.21}
\end{equation}
Similarly, we use the condition \eqref{1.20} to obtain 
\[
D_2(\bfx_0,r)\ =\ \int_r^{\infty} \mu\bigl( B_r(\bfx_0) \bigr)\; 
\frac{\rmd\zeta}{\zeta^2}\
=\ \frac{1}{r} \int_{B_r(\bfx_0)} f(\bfx)\; \rmd\bfx\
\le\ c \frac{1}{(\ln r^{-1})^{1+\gamma}}.
\]
Thus, we get 
\begin{equation}
\lim_{r\to 0+}\ \sup_{\bfx_0\br\in\br\R^3}\ D_2(\bfx_0,r)\ 
\le\ c \lim_{r\to 0+}\ \frac{1}{(\ln r^{-1})^{1+\gamma}}\ =\ 0.
\label{1.22}
\end{equation}
From \eqref{1.21} and \eqref{1.22} we get the desired result.
\end{proof}

\begin{corollary} \label{C1}
Let $\gamma > 0$ and $f$ a nonnegative measurable function 
in $\R^3\times\cT$, where $\cT\subset\R$. 
If 
\[
\sup_{t\in\cT} \|\cF_{\gamma}(f(\, .\, ,t))\|_{3/2} < \infty,
\]
then
\[
\lim_{r\to 0+}\ \sup_{t\in\cT} \sup_{\bfx_0\br\in\br\R^3}\ 
\int_{B_r(\bfx_0)}
\frac{f(\bfx,t)}{|\bfx-\bfx_0|}\; \rmd\bfx\ =\ 0.
\]
\end{corollary}

\begin{proof}
As Remark \ref{R1} an upper bound in \eqref{1.10} may depends 
only on the norm $\|\cF_{\gamma}(f(\, .\, ,t))\|_{3/2}$, which is 
uniformly bounded in $t \in \cT$.
Thus, we conclude the desired result from Lemma \ref{L1.1} and Lemma \ref{L1.2}.
\end{proof}


\section{The proof of Theorem \ref{T1} under condition a).} 
\label{S2}

In this section, we suppose that $(\bfu,\bfb,p)$ is a suitable
weak solution to the MHD initial-value problem
(\ref{1.1})--(\ref{1.4}) in $Q_T$, satisfying condition a) of
Theorem \ref{T1}. We will prove that $(\bfu,\bfb,p)$ has no
singular points in $\R^3\times(0,T_0]$. Assume, by
contradiction, that $\cS_{(0,T_0]}\not=\emptyset$. Due to the
assumption that $\bfu_0,\bfb_0\in \Ws^{1,2}(\R^3)$, there
exists $T_1\in(0,T_0)$ such that $\cS_{(0,T_1)}=\emptyset$.
Thus, the first time instant, when a singular point appears is
a point from $[T_1,T_0]$. Let us denote this time instant by
$t_0$. In accordance with the terminology from \cite{Ga1}, we
may call it \emph{epoch of irregularity}. (Recall that,
generally, $t_0\in(0,T)$ is said to be an epoch of irregularity
of the solution $(\bfu,\bfb,p)$ if there exists $\delta>0$ such
that $\cS(t)=\emptyset$ for all $t\in(t_0-\delta,t_0)$ and
$\cS(t_0)\not=\emptyset$.)

\subsection{More on formula (\ref{1.5}).} \label{SS2.1}
Obviously, the right hand side of formula (\ref{1.5}) has a
sense at every point $(\bfx,t)\in\R^3\times(0,T)$ such that
$\cS(t)=\emptyset$. Let us show that it also has a sense at
all regular points $(\bfx,t)$ of the solution $(\bfu,\bfb,p)$
which lie on the time level $t$ such that
$\cS(t)\not=\emptyset$. Thus, let $d>0$ and $\bfx$ be a point
in $\R^3$ whose distance from $\cS(t)$ is greater than or
equal to $2d$. Splitting the integral on the right hand side of
(\ref{1.5}) to the sum of the integral over $B_{d}(\bfx)$ and
the integral over $\R^3\smallsetminus B_{d}(\bfx)$ and applying
twice the integration by parts to the integral over
$\R^3\smallsetminus B_{d}(\bfx)$, we obtain
\begin{align*}
&\int_{\R^3} \frac{1}{|\bfx-\bfy|}\ 
\divg\divg\bigl[\bfu(\bfy,t)\otimes\bfu(\bfy,t)-\bfb(\bfy,t)\otimes
\bfb(\bfy,t)\bigr]\; \rmd\bfy \\
&=\ \int_{\R^3} \frac{1}{|\bfx-\bfy|}\ \frac{\partial^2}{\partial
y_i\, \partial y_j}\br\bigl[ u_i(\bfy,t)\br
u_j(\bfy,t)-b_i(\bfy,t)\br b_j(\bfy,t) \bigr]\; \rmd\bfy
\\ \noalign{\vskip 6pt}
&=\ I_d^{(1)}(\bfx,t)+ I_d^{(2)}(\bfx,t),
\end{align*}
where
\begin{align}
I_d^{(1)}(\bfx,t)\ 
&=\ \int_{B_{d}(\bfx)} \frac{1}{|\bfx-\bfy|}\
\frac{\partial^2}{\partial y_i\, \partial y_j}\br\bigl[
u_i(\bfy,t)\br u_j(\bfy,t)-b_i(\bfy,t)\br
b_j(\bfy,t) \bigr]\; \rmd\bfy \nonumber \\
&\quad +\int_{S_{d}(\bfx)}
\frac{n_i^{\bfx}}{|\bfx-\bfy|}\ \frac{\partial}{\partial
y_j}\br\bigl[ u_i(\bfy,t)\br
u_j(\bfy,t)-b_i(\bfy,t)\br b_j(\bfy,t) \bigr]\; \rmd_{\bfy}S \nonumber \\
&\quad -\int_{S_{d}(\bfx)} \frac{\partial}{\partial
y_i}\br\Bigl( \frac{1}{|\bfx-\bfy|} \Bigr)\, n_j^{\bfx}\, \bigl[
u_i(\bfy,t)\br u_j(\bfy,t)-b_i(\bfy,t)\br b_j(\bfy,t) \bigr]\;
\rmd_{\bfy}S, \label{2.1} 
\end{align}
and 
\begin{align}
I_d^{(2)}(\bfx,t)\ &=\ \int_{\R^3\smallsetminus B_{d}(\bfx)}
\frac{\partial^2}{\partial y_i\, \partial y_j}\,
\Bigl(\frac{1}{|\bfy-\bfx|}\Bigr)\ \bigl[ u_i(\bfy,t)\br
u_j(\bfy,t)-b_i(\bfy,t)\br b_j(\bfy,t) \bigr]\;
\rmd\bfy \nonumber \\
&=\ \int_{\R^3\smallsetminus B_{d}(\bfx)} \bbK(\bfy-\bfx):
\bigl[\bfu(\bfy,t)\otimes\bfu(\bfy,t)-\bfb(\bfy,t)\otimes
\bfb(\bfy,t)\bigr]\; \rmd\bfy. \label{2.2}
\end{align}
Here, $S_{d}(\bfx)$ is the sphere with center $\bfx$ and radius
$d$,
\begin{displaymath}
\bfn^{\bfx}(\bfy)\ \equiv\ \bigl( n_1^{\bfx}(\bfy),
n_2^{\bfx}(\bfy), n_3^{\bfx}(\bfy) \bigr)\ :=\
\frac{\bfx-\bfy}{|\bfx-\bfy|}\ =\ \frac{\bfx-\bfy}{d}
\end{displaymath}
and $\bbK(\bfy-\bfx):=\nabla_{\!\bfy}^2\br|\bfy-\bfx|^{-1}$ is the
second order tensor with the entries
\begin{displaymath}
k_{ij}(\bfy-\bfx) = \frac{\partial^2}{\partial y_i\,
\partial y_j}\, \Bigl(\frac{1}{|\bfy-\bfx|}\Bigr) =
-\frac{\partial}{\partial y_i}\, \frac{y_j-x_j}{|\bfy-\bfx|^3}\ =\
3\, \frac{(y_i-x_i)(y_j-x_j)}
{|\bfy-\bfx|^5}-\frac{\delta_{ij}}{|\bfy-\bfx|^3}
\end{displaymath}
for $i,j=1,2,3$. As all integrals in $I_d^{(1)}(\bfx,t)$ and
$I_d^{(2)}(\bfx,t)$ converge, (\ref{1.5}) makes sense. Thus, since
$d>0$ can be chosen arbitrarily small, the pressure is defined by
formula (\ref{1.5}) at every regular point of $(\bfu,\bfb,p)$.

\subsection{An estimate of $\bfu$ and $\bfb$ in the
neighborhood of infinity.} \label{SS2.2} If a suitable weak
solution $(\bfu,\bfb,p)$ is regular in
$\R^3\times(t_0-\delta,t_0)$, where $0\leq t_0-\delta<t_0\leq
T$, then
\begin{equation}
\lim_{R\to\infty}\ \sup_{t_0-\delta\leq t\leq t_0}\
\int_{\R^3\smallsetminus B_R(\bfzero)} \bigl( |\bfu(\bfx,t)|^2+
|\bfb(\bfx,t)|^2 \bigr)\; \rmd\bfx\ =\ 0. \label{2.3}
\end{equation}
The same formula is proven in \cite{SeSve} just for a suitable
weak solution $\bfu$ of the Navier--Stokes equations. (See formula
(4.6) in \cite{SeSve}.) The derivation uses the subtraction of the
generalized (i.e.~localized) energy equality from the energy
equality for solution $\bfu$, and on appropriate estimates of the
difference. The presence of function $\bfb$, as an additional
component of the solution, affects the whole procedure only
technically. The formula, in the complete form (\ref{2.3})
(i.e.~for a suitable weak solution $(\bfu,\bfb,p)$ for the MHD
equations), is also used in the proof of Theorem 1.3 in
\cite{KaLee3}.

\subsection{Important identities.} \label{SS2.3} Let $\bfx_0\in\R^3$,
$t\in(0,T)$, $R>0$ and $\alpha\in[0,1]$. By (\ref{1.5}), we have
\begin{align}
& \int_{B_R(\bfx_0)} |\bfx_0-\bfy|^{-\alpha}\, \Bigl[
p(\bfy,t)+\frac{1}{2}\, |\bfb(\bfy,t)|^2\Bigr]\; \rmd\bfy
\nonumber \\
&=\ \frac{1}{4\pi}\int_{\R^3}\divg\divg
\bigl[\bfu(\bfx,t)\otimes\bfu(\bfx,t)-\bfb(\bfx,t)\otimes
\bfb(\bfx,t)\bigr]\, \biggl( \int_{B_R(\bfx_0)}
\frac{|\bfx_0-\bfy|^{-\alpha}}{|\bfx-\bfy|}\; \rmd\bfy \biggr)\,
\rmd\bfx \nonumber \\
&=\ \frac{1}{4\pi}\int_{\R^3}
\bigl[\bfu(\bfx,t)\otimes\bfu(\bfx,t)-\bfb(\bfx,t)\otimes
\bfb(\bfx,t)\bigr]\ \nabla_{\bfx}^2\biggl( \int_{B_R(\bfx_0)}
\frac{|\bfx_0-\bfy|^{-\alpha}}{|\bfx-\bfy|}\; \rmd\bfy \biggr)\,
\rmd\bfx \label{2.4}
\end{align}
One can compute the integral to see that 
\begin{equation}
\nabla_{\bfx}^2 \int_{B_R(\bfx_0)}
\frac{|\bfx_0-\bfy|^{-\alpha}}{|\bfx-\bfy|}\; \rmd\bfy\ =\
\frac{4\pi\, |\bfx-\bfx_0|^{-\alpha}}{3-\alpha}\, \biggl(-\bbI+
\alpha\, \frac{(\bfx-\bfx_0)\otimes(\bfx-
\bfx_0)}{|\bfx-\bfx_0|^2} \biggr) \label{2.5}
\end{equation}
for $|\bfx-\bfx_0|\leq R$ and
\begin{equation}
\nabla_{\bfx}^2 \int_{B_R(\bfx_0)}
\frac{|\bfx_0-\bfy|^{-\alpha}}{|\bfx-\bfy|}\; \rmd\bfy\ =\
\frac{4\pi\, R^{3-\alpha}}{3-\alpha}\, |\bfx-\bfx_0|^{-3}\,
\biggl(-\bbI+3\, \frac{(\bfx-\bfx_0)\otimes(\bfx-\bfx_0)}
{|\bfx-\bfx_0|^2} \biggr) \label{2.6}
\end{equation}
for $|\bfx-\bfx_0|>R$. As the derivation of (\ref{2.5}) and
(\ref{2.6}) is quite technical, we provide its details in
Appendix. Substituting formulas (\ref{2.5}) and (\ref{2.6}) to
(\ref{2.4}), we obtain
\begin{align}
&\int_{B_R(\bfx_0)} \frac{1}{|\bfx-\bfx_0|^{\alpha}}\; \Bigl(
p+\frac{1}{2}\, |\bfb|^2 \Bigr)\; \rmd\bfx \nonumber \\
&=\ \frac{1}{3-\alpha}\int_{B_R(\bfx_0)}\frac{1}
{|\bfx-\bfx_0|^{\alpha}}\, \bigl( -|\bfu|^2+|\bfb|^2+\alpha\,
|\bfu_{\rm r}^{\bfx_0}|^2-\alpha\, |\bfb_{\rm
r}^{\bfx_0}|^2\bigr)\;
\rmd\bfx \nonumber \\
&\quad +\frac{1}{3-\alpha}\int_{\R^3\smallsetminus
B_R(\bfx_0)}\frac{R^{3-\alpha}}{|\bfx-\bfx_0|^3}\; \bigl(
2\br|\bfu_{\rm r}^{\bfx_0}|^2-|\bfu_{\rm p}^{(\bfx_0}|^2-
2\br|\bfb_{\rm r}^{\bfx_0}|^2+|\bfb_{\rm p}^{\bfx_0}|^2 \bigr)\;
\rmd\bfx, \label{2.7}
\end{align}
where
\begin{align*}
& \bfur(\bfx,t):=\Bigl( \frac{\bfu(\bfx,t)\cdot(\bfx-\bfx_0)}
{|\bfx-\bfx_0|} \Bigr)\, \frac{\bfx-\bfx_0}{|\bfx-\bfx_0|}, &&
\bfup(\bfx,t) :=\bfu(\bfx,t)-\bfur(\bfx,t), \\ \noalign{\vskip
2pt}
& \bfbr(\bfx,t):=\Bigl( \frac{\bfb(\bfx,t)\cdot(\bfx-\bfx_0)}
{|\bfx-\bfx_0|} \Bigr)\, \frac{\bfx-\bfx_0}{|\bfx-\bfx_0|}, &&
\bfbp(\bfx,t) :=\bfb(\bfx,t)-\bfbr(\bfx,t).
\end{align*}
Note that $\bfur(\bfx,t)$ is the orthogonal projection of
$\bfu(\bfx,t)$ to the ``radial'' direction $\bfx-\bfx_0$ (radial
in the coordinate system centered at the point $\bfx_0$) and
$\bfup(\bfx,t)$ is the orthogonal projection of $\bfu(\bfx,t)$ to
the plane perpendicular to $\bfx-\bfx_0$. The same explanation
also holds for $\bfbr(\bfx,t)$ and $\bfbp(\bfx,t)$. Equality
(\ref{2.7}) yields
\begin{align*}
&\frac{1}{R}\int_{B_R(\bfx_0)} 
\frac{R^{\alpha}}{|\bfx-\bfx_0|^{\alpha}}\, \Bigl[ (3-\alpha)\,
p+|\bfup|^2+(1-\alpha)\, |\bfur|^2+ \frac{1-\alpha}{2}\, |\bfbp|^2 
+\frac{1+\alpha}{2}\, |\bfbr|^2 \Bigr]\; \rmd\bfx \\
&=\ \int_{\R^3\smallsetminus B_R(\bfx_0)}\frac{R^2}
{|\bfx-\bfx_0|^3}\; \bigl( 2\, |\bfur|^2- |\bfup|^2-2\, |\bfbr|^2+
|\bfbp|^2 \bigr)\; \rmd\bfx. 
\end{align*}
Particularly, choosing $\alpha=1$ and $\alpha=0$, we get
\begin{align}
&\int_{B_R(\bfx_0)} \frac{1}{|\bfx-\bfx_0|}\ \bigl(2p+|\bfup|^2+
|\bfbr|^2\bigr)\; \rmd\bfx \nonumber \\
\noalign{\vskip 2pt}
&=\ \int_{B_R(\bfx_0)}\frac{1}{R}\ \Bigl(3p
+|\bfu|^2+\frac{1}{2}\, |\bfb|^2\Bigr)\; \rmd\bfx \nonumber \\
\noalign{\vskip 2pt}
&=\ \int_{\R^3\smallsetminus B_R(\bfx_0)}
\frac{R^2}{|\bfx-\bfx_0|^3}\, \bigl[\br 2\br|\bfur|^2-|\bfup|^2-
2\br|\bfbr|^2+|\bfbp|^2 \br\bigr]\; \rmd\bfx.
\label{2.8}
\end{align}

\subsection{The continuity of $\bfu$ and $\bfb$ from $(0,t_0]$
to $\bfL^2(\R^3)$.} \label{SS2.4} As the solution
$(\bfu,\bfb,p)$ has no singular points in $\R^3\times(0,t_0)$,
the norms $\|\bfu(\, .\, ,t)\|_2$ and $\|\bfb(\, .\, ,t)\|_2$
depend continuously on $t$ for $t\in(0,t_0)$. Our next aim in
this subsection is to prove that
\begin{equation}
\lim_{t\to t_0-}\ \int_{\R^3} \bigl(|\bfu(\bfx,t)-
\bfu(\bfx,t_0)|^2+|\bfb(\bfx,t)-\bfb(\bfx,t_0)|^2\bigr)\;
\rmd\bfx\ =\ 0. \label{2.9}
\end{equation}
We shall use the next lemma:

\begin{lemma} \label{L2.1}
If $\Omega_0\subset\R^3$ is a bounded domain and 
$0 < \delta < t_0 < T$, then the following implications hold:
\begin{align}
\sup_{R>0,\ \bfx_0\in\Omega_0}\ \frac{1}{R}\ \
{\displaystyle\mathrel{\esssup_{t_0-\delta<t<t_0}}}\ 
\|\bfu(\, .\, ,t)\|_{2;\, B_R(\bfx_0)}^2 < \infty\ 
&\implies\ \lim_{t\to t_0-}\ \|\bfu(\, .\, ,t)-\bfu(\, .\,
,t_0)\|_{2;\, \Omega_0} = 0, \label{2.10} \\
\sup_{R>0\ \bfx_0\in\Omega_0}\ \frac{1}{R}\ \
{\displaystyle\mathrel{\esssup_{t_0-\delta<t<t_0}}}\ 
\|\bfb(\, .\, ,t)\|_{2;\, B_R(\bfx_0)}^2 < \infty\ 
&\implies\ \lim_{t\to t_0-}\ \|\bfb(\, .\, ,t)-\bfb(\, .\,
,t_0)\|_{2;\, \Omega_0} = 0. \label{2.11}
\end{align}
\end{lemma}

\begin{proof}
As the function $\bfu$ is weakly continuous from
$(t_0-\delta,t_0]$ to $\bfL^2(\R^3)$, it is also weakly continuous
from $(t_0-\delta,t_0]$ to $\bfL^2(B_R(\bfx_0))$. Hence, due to
the lower semi-continuity of the norm in $\bfL^2(B_R(\bfx_0))$, we
have
\begin{equation}
{\displaystyle\mathrel{\esssup_{t_0-\delta<t<t_0}}}\
\|\bfu(\, .\, ,t)\|_{2;\, B_R(\bfx_0)}^2\ =\
\sup_{t_0-\delta<t\leq t_0}\ \|\bfu(\, .\, ,t)\|_{2;\,
B_R(\bfx_0)}^2. \label{2.12}
\end{equation}
Then the implication (\ref{2.10}) follows from Lemma 3.2 in
\cite{SeSve}. Note that the authors of \cite{SeSve} prove an
analogous implication in their Lemma 3.2, considering the
supremum over $R>0$, $\bfx_0\in\Omega_0$ and $t\in(0,t_0]$ in
the premise. However, due to (\ref{2.12}), the suprema on the
left hand side of (\ref{2.10}) are equal to just one supremum
over $R>0$, $\bfx_0\in\Omega_0$ and $t\in(0,t_0]$, which means
that Lemma 3.2 from \cite{SeSve} can be applied.

The validity of the implication (\ref{2.11}) can be confirmed in
the same way.

Note that an analogue of Lemma 3.2 from paper \cite{SeSve}, which
deals just with the Navier--Stokes equations, can also be found in
paper \cite{KaLee3}, which concerns the MHD equations.
\end{proof}

In order to prove (\ref{2.9}), let us at first show that the
premises in the implications (\ref{2.10}) and (\ref{2.11}) in Lemma
\ref{L2.1} are satisfied.

Since $\cF_{\gamma}(p_-)\in L^{\infty}(0,T;\, L^{3/2}(\R^3))$,
there exists set $\cT\subset(0,t_0)$ of 1D Lebesgue measure
zero such that the norm $\|\cF_{\gamma}(p_-(\, .\,
,t))\|_{3/2}$ is uniformly bounded for
$t\in(0,t_0)\smallsetminus\cT$. Then, due to Corollary
\ref{C1}, there exists $R_0>0$ such that for all $t\in(0,t_0)\smallsetminus\cT$,
\[
\sup_{R\in(0,R_0)} \int_{B_R(\bfx_0)}\frac{p_-(\bfx,t)}{|\bfx-\bfx_0|}\ \leq 1.
\]

Let $\bfx_0\in\R^3$ and $R\in(0,R_0)$. It follows from the second
identity in (\ref{2.8}) that at each time
$t\in(0,t_0)\smallsetminus\cT$, we have
\begin{align*}
&\frac{1}{R} \int_{B_R(\bfx_0)} \Bigl( |\bfu|^2+\frac{1}{2}\,
|\bfb|^2\Bigr)\; \rmd\bfx \\
&\leq\ \frac{1}{R} \int_{B_R(\bfx_0)}
\Bigl( |\bfu|^2+\frac{1}{2}\, |\bfb|^2+3\br\bigl[
p+p_-\bigr] \Bigr)\; \rmd\bfx \\
&=\ \int_{B_R(\bfx_0)} \frac{1}{|\bfx-\bfx_0|}\ \bigl(2\br
p+|\bfup|^2+ |\bfbr|^2\bigr)\;
\rmd\bfx+\frac{3}{R}\int_{B_R(\bfx_0)}p_-\; \rmd\bfx \\
&=\ \int_{B_R(\bfx_0)} \frac{1}{|\bfx-\bfx_0|}\ \bigl(2\br
p_++|\bfup|^2+ |\bfbr|^2\bigr)\; \rmd\bfx+
\frac{3}{R}\int_{B_R(\bfx_0)}p_-\; \rmd\bfx 
-\int_{B_R(\bfx_0)}\frac{2\br p_-}{|\bfx-\bfx_0|}\; \rmd\bfx,
\end{align*}
which is bounded by 
\begin{align*}
&\int_{B_{R_0}(\bfx_0)} \frac{1}{|\bfx-\bfx_0|}\ \bigl(2\br
p+|\bfup|^2+ |\bfbr|^2\bigr)\; \rmd\bfx+
\frac{3}{R}\int_{B_R(\bfx_0)}p_-\; \rmd\bfx 
+ \int_{B_{R_0}(\bfx_0)\smallsetminus
B_R(\bfx_0)}\frac{2\br p_-}{|\bfx-\bfx_0|}\; \rmd\bfx \\
&\leq\ \int_{\R^3\smallsetminus B_{R_0}(\bfx_0)}
\frac{R^2}{|\bfx-\bfx_0|^3}\, \bigl[\br 2\br|\bfu_{\rm
r}|^2-|\bfu_{\rm p}|^2- 2\br|\bfb_{\rm r}|^2-|\bfb_{\rm p}|^2
\br\bigr]\; \rmd\bfx \nonumber \\
&\quad +\frac{3}{R}\int_{B_R(\bfx_0)}p_-\;
\rmd\bfx+\int_{B_{R_0}(\bfx_0)\smallsetminus
B_R(\bfx_0)}\frac{2\br p_-}{|\bfx-\bfx_0|}\; \rmd\bfx
\nonumber \\
&\leq\ \frac{c}{R_0}\int_{R^3\smallsetminus B_{R_0}(\bfx_0)}
\bigl( |\bfu|^2+|\bfb|^2\bigr)\; \rmd\bfx+
\int_{B_{R_0}(\bfx_0)}\frac{3\br p_-}{|\bfx-\bfx_0|}\;
\rmd\bfx \nonumber \\
&\leq\ \frac{c}{R_0}\, \bigl(\|\bfu(\, .\, ,t)\|_2^2+\|\bfb(\, .\,
,t)\|_2^2\bigr)+3,
\end{align*}
where the constant $c$ is independent of $\bfx_0$, $t$, $R$ and $R_0$.

We have shown that the terms $R^{-1}\, \|\bfu(\, .\, ,t)\|_{2;\,
B_R(\bfx_0)}$ and $R^{-1}\, \|\bfb(\, .\, ,t)\|_{2;\,
B_R(\bfx_0)}$ are bounded above and the bound is independent of
$\bfx_0$, $t$ and $R$ for $\bfx_0\in\R^3$,
$t\in(0,t_0)\smallsetminus\cT$ and $R\in(0,R_0)$. If $R>R_0$ then,
obviously, $R^{-1}\, \|\bfu(\, .\, ,t)\|_{2;\, B_R(\bfx_0)}\leq
R_0^{-1}\, \|\bfu(\, .\, ,t)\|_{2;\, \R^3}\leq c/R_0$, where the constant $c$
is independent of $\bfx_0$, $t$, $R$ and $R_0$. (The same
estimates also hold for function $\bfb$.) This shows that the
premises in the implications (\ref{2.10}) and (\ref{2.11}) are true,
and the suprema on the left hand sides of (\ref{2.10}) and
(\ref{2.11}) can be even considered over all $\bfx_0\in\R^3$ and
not only over $\bfx_0$ from a bounded domain $\Omega_0$. Thus, the
statements of the implications are also true for any bounded
domain $\Omega_0$ in $\R^3$. Combining this result with
(\ref{2.3}), we obtain (\ref{2.9}).

\subsection{The continuity of $I_d^{(2)}$ in
$\R^3\times(0,t_0]$.} \label{SS2.5} Let $d>0$. Recall that the
function $I_d^{(2)}$ is defined in (\ref{2.2}). 
Assume that a sequence of points $\{(\bfx_n,t_n)\}$ in
$\R^3\times(0,t_0]$ converges to a point
$(\bfx_*,t_*)\in\R^3\times(0,t_0]$ as $n\to\infty$. 
Obviously,
\begin{align}
\bigl| I_d^{(2)}(\bfx_n,t_n) - I_d^{(2)}(\bfx_*,t_*)
\bigr| 
\leq \bigl| I_d^{(2)}(\bfx_n,t_n)-I_d^{(2)}(\bfx_n,t_*)
\bigr|+\bigl| I_d^{(2)}(\bfx_n,t_*)-I_d^{(2)}(\bfx_*,t_*) \bigr|.
\label{2.13}
\end{align}
For simplicity, we denote by
$\bbU(\bfy,t)$ and $\bbB(\bfy,t)$ the expressions $\bfu(\bfy,t)\otimes
\bfu(\bfy,t)$ and $\bfb(\bfy,t)\otimes\bfb(\bfy,t)$, respectively. Then the first
term on the right hand side can be estimated as follows:
\begin{align*}
&\bigl| I_d^{(2)} (\bfx_n,t_n)-I_d^{(2)}(\bfx_n,t_*) \bigr| \\
&=\ \biggl| \int_{\R^3\smallsetminus B_{d}(\bfx_n)}
\bbK(\bfy-\bfx_n):\bigl[ \bbU(\bfy,t_n)-\bbU(\bfy,t_*)-
\bbB(\bfy,t_n)-\bbB(\bfy,t_*) \bigr]\; \rmd\bfy \biggr| \\
&\leq\ \frac{c}{d^3}\ \int_{\R^3} \Bigl( \bigl|
\bbU(\bfy,t_n)-\bbU(\bfy,t_*) \bigr|+\bigl|
\bbB(\bfy,t_n)-\bbB(\bfy,t_*) \bigr| \Bigr)\; \rmd\bfy.
\end{align*}
Since $t_*\in(0,t_0]$, the right hand side tends to zero for
$n\to\infty$ due to the continuity of $\|\bfu(\, .\, ,t)\|_2$ and
$\|\bfb(\, .\, ,t)\|_2$ for $t\in(0,t_0]$. The second term on the
right hand side of (\ref{2.13}) can be estimated in the following way:
\begin{align}
& \bigl| I_d^{(2)}(\bfx_n,t_*)-I_d^{(2)}(\bfx,t_*) \bigr|
\nonumber \\
&=\ \biggl| \int_{\R^3\smallsetminus B_{d}(\bfx_n)}
\bbK(\bfy-\bfx_n):\bigl[\bbU(\bfy,t_*)-
\bbB(\bfy,t_*)\bigr]\; \rmd\bfy \nonumber \\
&\quad -\int_{\R^3\smallsetminus B_{d}(\bfx_*)}
\bbK(\bfy-\bfx_*):\bigl[\bbU(\bfy,t_*)-
\bbB(\bfy,t_*)\bigr]\; \rmd\bfy \biggr| \nonumber \\
&\leq\ \biggl| \biggl( \int_{\R^3\smallsetminus
B_{d}(\bfx_n)}-\int_{\R^3\smallsetminus B_{d}(\bfx_*)} \biggr)\
\bbK(\bfy-\bfx_n):\bigl[\bbU(\bfy,t_*)-
\bbB(\bfy,t_*)\bigr]\; \rmd\bfy \biggr| \nonumber \\
&\quad +\biggl| \int_{\R^3\smallsetminus B_{d}(\bfx_*)}
\bigl[\bbK(\bfy-\bfx_n)-\bbK(\bfy-\bfx_*)\bigr]:\bigl[\bbU(\bfy,t_*)
-\bbB(\bfy,t_*)\bigr]\; \rmd\bfy \biggr| \nonumber \\
&=\ \biggl| \biggl(
\int_{B_{d}(\bfx_*)\smallsetminus
B_{d}(\bfx_n)}-\int_{B_{d}(\bfx_n)\smallsetminus B_{d}(\bfx_*)}
\biggr)\ \bbK(\bfy-\bfx_n):\bigl[\bbU(\bfy,t_*)-
\bbB(\bfy,t_*)\bigr]\; \rmd\bfy \biggr| \nonumber \\
&\quad +\biggl| \int_{\R^3\smallsetminus B_{d}(\bfx_*)}
\bigl[\bbK(\bfy-\bfx_n)-\bbK(\bfy-\bfx_*)\bigr]:\bigl[\bbU(\bfy,t_*)
-\bbB(\bfy,t_*)\bigr]\; \rmd\bfy \biggr|. \label{2.14}
\end{align}
If $n$ is so large that $|\bfx_n-\bfx_*|<d$ then the first modulus
on the right hand side is less than or equal to
\begin{align*}
\frac{c}{d^3}\ \biggl(\int_{B_{d}(\bfx_*)\smallsetminus
B_{d}(\bfx_n)}+\int_{B_{d}(\bfx_n)\smallsetminus B_{d}(\bfx_*)}
\biggr)\ \bigl|\bbU(\bfy,t_*)- \bbB(\bfy,t_*)\bigr|\; \rmd\bfy,
\end{align*}
where the constant $c$ is independent of $n$. This tends to zero for
$n\to\infty$ because both $\bbU(\, .\, ,t_*)$ and $\bbB(\, .\,
,t_*)$ are in $L^1(\R^3)^{3\times 3}$ and the measures of
$B_{d}(\bfx_*)\smallsetminus B_{d}(\bfx_n)$ and
$B_{d}(\bfx_n)\smallsetminus B_{d}(\bfx_*)$ tend to zero as
$n\to\infty$. In order to show that the second modulus on the
right hand side of (\ref{2.14}) also tends to zero as
$n\to\infty$, consider $n$ so large that
$|\bfx_n-\bfx_*|<\frac{1}{2}\br d$. Then
$|\bfy-\bfx_n|\geq\frac{1}{2}\br d$ for $\bfy\in\R^3\smallsetminus
B_{d}(\bfx_*)$. Obviously, for these $y$, the inequality
$|\bfy-\bfx_*|\geq d$ also holds true. Hence
\begin{align*}
\bigl| k_{ij}(\bfy & -\bfx_n)-k_{ij}(\bfy-\bfx_*) \bigr|\ =\
\biggl| \frac{\partial^2}{\partial y_i\, \partial y_j}\, \Bigl(
\frac{1}{|\bfy-\bfx_n|}-\frac{1}{|\bfy-\bfx_*|} \Bigr) \biggr|\
\leq\ c\, \frac{|\bfx_n-\bfx_*|}{d^4},
\end{align*}
where the constant $c$ is independent of $d$, $\bfx_n$ and $\bfx_*$. Thus, the
second modulus on the right hand side of (\ref{2.14}) is bounded
above by
\begin{displaymath}
c\, \frac{|\bfx_n-\bfx_*|}{d^4}\int_{\R^3}\, \bigl|
\bbU(\bfy,t_*)-\bbB(\bfy,t_*) \bigr|\; \rmd\bfy,
\end{displaymath}
which tends to zero as $n\to\infty$. Thus, we have shown that for
each $d>0$, the function $I_d^{(2)}$ is continuous on
$\R^3\times(0,t_0]$.

\subsection{The continuity of $p$ and $p_-$ in
$\bigl(\R^3\times(0,t_0]\bigr)\smallsetminus\cS_{(0,t_0]}$.}
\label{SS2.6} Recall that points of $\cS_{(0,t_0]}$ may appear
in $\R^3\times(0,t_0]$ only on the time level $t=t_0$, which
means $\cS_{(0,t_0]}=\cS_{\{t_0\}}$. Function $p$
satisfies
\begin{displaymath}
p(\bfx,t)\ =\ \frac{1}{4\pi}\
\bigl[ I_d^{(1)}(\bfx,t)+ I_d^{(2)}(\bfx,t) \bigr],
\end{displaymath}
where $I_d^{(1)}$ and $I_d^{(2)}$ are the functions, defined by
(\ref{2.1}) and (\ref{2.2}). Also recall that we have already
proven the continuity of $I_d^{(2)}$ in $\R^3\times(0,t_0]$ for
any $d>0$ in subsection \ref{SS2.5}. We still need to show that
$I_d^{(1)}$ is continuous in $\bigl(\R^3\times(0,t_0]\bigr)
\smallsetminus\cS_{\{t_0\}}$.

Function $I_d^{(1)}$ is continuous at each point $(\bfx,t)\in
\R^3\times(0,t_0]$, whose distance from $\cS_{\{t_0\}}$ is
greater than or equal to $2d$, due to the H\"older continuity
of $\bfu$ and $\bfb$ in $B_{2d}((\bfx,t))$ (the ball in
$\R^4$). Hence the same statement on continuity can also be
made on $I_d^{(1)}+I_d^{(2)}$. However, as the sum
$I_d^{(1)}+I_d^{(2)}$ is independent of $d$, because it equals
$4\pi p$, it is a continuous function on the whole set
$\bigl(\R^3\times (0,t_0]\bigr)\smallsetminus\cS_{\{t_0\}}$.

Consequently, both $p$ and $p_-$ are continuous functions in
$\bigl(\R^3\times (0,t_0]\bigr)\smallsetminus\cS_{\{t_0\}}$.

\subsection{The boundedness of $\|\cF_{\gamma}(p_-(\, .\,
,t))\|_{3/2}$ up to the epoch of irregularity $t_0$.}
\label{SS2.7} Let $\Omega_0$ be a bounded domain in $\R^3$ and
$d>0$. Since $\cS(t_0)$ is a closed subset of $\R^3$ of 1-dimensional
Hausdorff measure zero, we have
\begin{equation}
\int_{\Omega_0}\cF_{\gamma}^{3/2}(p_-(\bfy,t_0))\; \rmd\bfy\ =\
\lim_{d\to 0+} \int_{\Omega_0\smallsetminus U_d(\cS(t_0))}
\cF_{\gamma}^{3/2}(p_-(\bfy,t_0))\; \rmd\bfy,
\label{2.15}
\end{equation}
which does not exclude that both sides are infinity.
(Here, we denote by $U_d(\cS(t_0))$ the $d$--neighborhood of
set $\cS(t_0)$ in $\R^3$.) As $p_-(\, .\, ,t_0)$ is continuous
on
$\cM_1:=\bigl(\R^3\times(0,t_0]\bigr)\smallsetminus\cS_{\{t_0\}}$
and $\cM_2:=\bigl(\overline{\Omega_0}\smallsetminus U_d(
\cS(t_0))\bigr)\times[t_0/2,t_0]$ is a bounded closed subset of
$\cM_1$, the function $\cF_{\gamma}(p_-)$ is uniformly
continuous on $\cM_2$. 
Hence if we put 
\[
c_2 := \esssup_{0<t<t_0}\ \|\cF_{\gamma}(p_-(\, .\, ,t))\|_{3/2}^{3/2},
\]
then 
\[
\int_{\Omega_0\smallsetminus
U_d(\cS(t_0))}\cF_{\gamma}^{3/2}(p_-(\bfy,t_0))\; \rmd\bfy\ 
=\ \lim_{t\to t_0-}\int_{\Omega_0\smallsetminus
U_d(\cS(t_0))}\cF_{\gamma}^{3/2}(p_-(\bfy,t_0))\; \rmd\bfy \\
\leq\ c_2.
\]
This shows that the integral in the limit on the right hand side
of (\ref{2.15}) is finite and bounded by $c_2$. 
Since $c_2$ is independent of $d$, the integral on the left hand side
of (\ref{2.15}) is also bounded by $c_2$. 
Since this holds for any bounded domain $\Omega_0$, the integral
$\int_{\R^3}\cF_{\gamma}^{3/2}(p_-(\bfy,t_0))\; \rmd\bfy$ is bounded by $c_2$ as well.

\subsection{Completion of the proof of Theorem \ref{T1} under
condition a).} \label{SS2.8} In order to deny the existence of
a singular point of the suitable weak solution $(\bfu,\bfb,p)$
on the time level $t=t_0$, we will use the next lemma, which is a special case of Theorem 1.1 in
\cite{KaLee3}.

\begin{lemma} \label{L2.2}
Let $(\bfu,\bfb,p)$ be a suitable weak solution to the MHD
initial-value problem (\ref{1.1})--(\ref{1.4}) in $Q_T$ and 
$(\bfx_0,t_0)\in Q_T$. There exists $\epsilon_*>0$ (independent of
the solution $(\bfu,\bfb,p)$ and the point $(\bfx_0,t_0)$) such
that if $B_{R_*}(\bfx_0)\times(t_0-R_*^2,t_0)\subset Q_T$ for some
$R_*>0$ and
\begin{align}
& \sup_{0<R<R_*}\ \ \sup_{t_0-R^2\leq t\leq t_0}\ \frac{1}{R}\
(\|\bfu(\, .\, ,t)\|_{2;\, B_R(\bfx_0)}^2\ <\ \epsilon_*,
\label{2.16} \\
& \sup_{0<R<R_*}\ \ \sup_{t_0-R^2\leq t\leq t_0}\ \frac{1}{R}\
\|\bfb(\, .\, ,t)\|_{2;\, B_R(\bfx_0)}^2\ <\ \infty
\label{2.17}
\end{align}
then $(\bfx_0,t_0)$ is a regular point of the solution $(\bfu,\bfb,p)$.
\end{lemma}

The proof of Theorem \ref{T1} under condition a) can now be
completed in this way.
Recall that, by assumption, the epoch of irregularity 
$t_0\in(0,T_0]$ is the first instant of time when a
singular point of the solution $(\bfu,\bfb,p)$ appears. Let
$\bfx_0\in\R^3$. Our aim is to show that there exists $R_*>0$ such
that (\ref{2.16}) and (\ref{2.17}) hold.

Note that due to Corollary \ref{C1} and the results of
subsection \ref{SS2.7}, we have for all $t\in(0,t_0]$,
\begin{equation}
\lim_{r\to 0+}\ \int_{B_r(\bfx_0)} \frac{p_-(\bfx,t)}
{|\bfx-\bf\bfx_0|}\; \rmd\bfx\ =\ 0. \label{2.18}
\end{equation}
Since the norm $\|\cF_{\gamma}(p_-(\, .\,
,t))\|_{3/2}$ is bounded as a function of $t$ on $(0,t_0]$, the
limit in (\ref{2.18}) is uniform with respect to $t\in(0,t_0]$.
Moreover, at time $t_0$, we also have
\begin{align}
&\int_{B_R(\bfx_0)} \frac{1}{|\bfx-\bfx_0|}\ \bigl[\br
|\bfup(\bfx,t_0)|^2+|\bfbr(\bfx,t_0)|^2+
2\br p_+(\bfx,t_0) \br\bigr]\; \rmd\bfx \nonumber \\
&=\ \int_{B_R(\bfx_0)} \frac{1}{|\bfx-\bfx_0|}\
\bigl[ |\bfup(\bfx,t_0)|^2+ |\bfbr(\bfx,t_0)|^2+ 2\br p(\bfx,t_0)
\bigr]\; \rmd\bfx 
+\int_{B_R(\bfx_0)}\frac{2\br p_-(\bfx,t_0)}
{|\bfx-\bfx_0|}\; \rmd\bfy \nonumber \\
&=\ \int_{\R^3\smallsetminus B_R(\bfx_0)}
\frac{R^2}{|\bfx-\bfx_0|^3}\  \bigl[\br 2\br|\bfu_{\rm
r}(\bfx,t)|^2-|\bfu_{\rm p}(\bfx,t)|^2- 2\br|\bfb_{\rm
r}(\bfx,t)|^2+|\bfb_{\rm p}(\bfx,t)|^2 \br\bigr]\;
\rmd\bfy \nonumber \\
&\quad + \int_{B_R(\bfx_0)}\frac{2\br
p_-(\bfx,t_0)}{|\bfx-\bfx_0|}\; \rmd\bfx\ \leq\ \frac{c}{R}+
\int_{B_R(\bfx_0)}\frac{2\br p_-(\bfx,t_0)}{|\bfx-\bfx_0|}\;
\rmd\bfx\ <\ \infty \label{2.19}
\end{align}
due to (\ref{2.8}) and (\ref{2.18}). Let $\epsilon_*$ be the
number in Lemma \ref{L2.2}.
Now, we choose $R_*>0$ so small that
\begin{equation}
\int_{B_{R_*}(\bfx_0)}\frac{2\br p_-(\bfx,t)}{|\bfx-\bfx_0|}\;
\rmd\bfx\ <\ \frac{\epsilon_*}{4} \label{2.20}
\end{equation}
for all $t\in(0,t_0]$ and
\[
\int_{B_{R_*}(\bfx_0)} \frac{1}{|\bfx-\bfx_0|}\ \bigl[\br
|\bfup(\bfx,t_0)|^2+|\bfbr(\bfx,t_0)|^2+ 2\br p_+(\bfx,t_0)
\br\bigr]\; \rmd\bfx\ <\ \frac{\epsilon_*}{4}.
\]
The latter is possible because the integral on the left hand
side of (\ref{2.19}) is finite and the integrand is nonnegative.
Then
\begin{align}
&\int_{\R^3\smallsetminus B_{R_*}(\bfx_0)} \frac{R_*^2}
{|\bfx-\bfx_0|^3}\ \bigl[\br 2\br|\bfu_{\rm r}(\bfx,t_0)|^2-
|\bfu_{\rm p}(\bfx,t_0)|^2-2\br|\bfb_{\rm r}(\bfx,t)|^2+
|\bfb_{\rm p}(\bfx,t_0)|^2 \br\bigr]\; \rmd\bfx \nonumber \\
&\equiv\ \int_{B_{R_*}(\bfx_0)}
\frac{1}{|\bfx-\bfx_0|}\ \bigl[\br
|\bfup(\bfx,t_0)|^2+|\bfbr(\bfx,t_0)|^2+
2\br p(\bfx,t_0) \br\bigr]\; \rmd\bfx \qquad \mbox{(by (\ref{2.8}))}
\nonumber \\
&\leq\ \int_{B_{R_*}(\bfx_0)}
\frac{1}{|\bfx-\bfx_0|}\ \bigl[\br
|\bfup(\bfx,t_0)|^2+|\bfbr(\bfx,t_0)|^2+ 2\br p_+(\bfx,t_0)
\br\bigr]\; \rmd\bfx\ <\ \frac{\epsilon_*}{4}. \label{2.21}
\end{align}
Applying (\ref{2.9}), we deduce that there exists a small 
positive number $\delta$ such that for $t\in(t_0-\delta,t_0]$,
\[
\int_{\R^3\smallsetminus B_{R_*}(\bfx_0)} \frac{R_*^2}
{|\bfx-\bfx_0|^3}\ \bigl[\br 2\br|\bfu_{\rm r}(\bfx,t)|^2-
|\bfu_{\rm p}(\bfx,t)|^2-2\br|\bfb_{\rm r}(\bfx,t)|^2+|\bfb_{\rm
p}(\bfx,t)|^2 \br\bigr]\; \rmd\bfx\ <\ \frac{\epsilon_*}{2}.
\]
Then, due to (\ref{2.8}) and
(\ref{2.20}), we also have for all $R\in(0,R_*)$ and on each time
level $t\in(t_0-\delta,t_0]$:
\begin{align*}
&\frac{1}{R} \int_{B_{R}(\bfx_0)} \Bigl( |\bfu|^2+ \frac{1}{2}\,
|\bfb|^2 \Bigr)\; \rmd\bfx\ \leq\ \frac{1}{R}\int_{B_{R}(\bfx_0)}
\Bigl( |\bfu|^2+\frac{1}{2}\, |\bfb|^2+3\br p_+\Bigr)\; \rmd\bfx \\
&=\ \frac{1}{R}\int_{B_R(\bfx_0)} \Bigl( |\bfu|^2+\frac{1}{2}\,
|\bfb|^2+3\br p \Bigr)\; \rmd\bfx+\frac{3}{R}\int_{B_R(\bfx_0)}
p_-\; \rmd\bfx \\
&=\ \int_{B_R(\bfx_0)} \frac{1}{|\bfx-\bfx_0|}\ \bigl[\br
|\bfup|^2+|\bfbr|^2+ 2\br p \br\bigr]\;
\rmd\bfx+\frac{3}{R}\int_{B_R(\bfx_0)} p_-\; \rmd\bfx \\
&=\ \int_{B_R(\bfx_0)} \frac{1}{|\bfx-\bfx_0|}\ \bigl[\br
|\bfup|^2+|\bfbr|^2+ 2\br p_+ \br\bigr]\;
\rmd\bfx-\int_{B_R(\bfx_0)} \frac{2\br p_-}{|\bfx-\bfx_0|}\;
\rmd\bfx 
+\frac{3}{R}\int_{B_R(\bfx_0)} p_-\; \rmd\bfx, 
\end{align*}
which is bounded by 
\begin{align*}
&\int_{B_{R_*}(\bfx_0)} \frac{1}{|\bfx-\bfx_0|}\ 
\bigl[\br |\bfup|^2+|\bfbr|^2+ 2\br p \br\bigr]\;
\rmd\bfx \\ 
&\quad + \int_{B_{R_*}(\bfx_0)} 
\frac{2\br p_-}{|\bfx-\bfx_0|}\; \rmd\bfx 
- \int_{B_R(\bfx_0)} \frac{2\br p_-}{|\bfx-\bfx_0|}\; \rmd\bfx 
+ \frac{3}{R}\int_{B_R(\bfx_0)} p_-\; \rmd\bfx \\
&=\ \int_{\R^3\smallsetminus B_{R_*}(\bfx_0)} \frac{R_*^2}
{|\bfx-\bfx_0|^3}\ \bigl[\br 2\br|\bfu_{\rm r}|^2- |\bfu_{\rm
p}|^2-2\br|\bfb_{\rm r}|^2- |\bfb_{\rm p}|^2 \br\bigr]\;
\rmd\bfx \\ 
&\quad + \int_{B_{R_*}(\bfx_0)} \frac{2\br p_-}{|\bfx-\bfx_0|}\;
\rmd\bfx - \int_{B_R(\bfx_0)} \frac{2\br p_-}{|\bfx-\bfx_0|}\; \rmd\bfx 
+ \frac{3}{R}\int_{B_R(\bfx_0)} p_-\; \rmd\bfx \\
&\leq\ \frac{\epsilon_*}{2}+\int_{B_{R_*}(\bfx_0)} \frac{3\br
p_-}{|\bfx-\bfx_0|}\; \rmd\bfx\ \leq\ \frac{\epsilon_*}{2}+
\frac{3\epsilon_*}{8}\ <\ \epsilon_*.
\end{align*}
As this holds independently of $R$ (for $R\in(0,R_*]$) and $t$
(for $t\in(t_0-\delta,t_0]$), we observe that (\ref{2.16}) and
(\ref{2.17}) hold. Thus, due to Lemma \ref{L2.2}, $(\bfx_0,t_0)$
is a regular point of the solution $(\bfu,\bfb,p)$. 
Since $\bfx_0$ was chosen arbitrarily in $\R^3$, the solution has no
singular points on the time level $t_0$. This is a contradiction
with the assumption that $t_0$ is an epoch of irregularity.
Consequently, the solution $(\bfu,\bfb,p)$ has no singular points
in $Q_T$. Using the results of \cite{MaNiSh}, we can state that
$\bfu$ and $\bfb$ are H\"older--continuous in $Q_T$. The proof of
Theorem \ref{T1} (under condition a)) is completed.


\section{The proof of Theorem \ref{T1} under condition b)}
\label{S3}

The subsections \ref{SS2.1}--\ref{SS2.5} can be repeated without
any changes. In subsection \ref{SS2.4} (on the left continuity
of $\bfu$ and $\bfb$ as functions of time in $(0,t_0]$), we
used condition a) of Theorem \ref{T1}. We show in the next
subsection \ref{SS3.1} that the same conclusion (formulated by
means of (\ref{2.9})) can also be proven if we consider
condition b) instead of condition a).

\subsection{The left continuity of $\bfu$ and $\bfb$ in the
$L^2$--norm at an epoch of irregularity.} \label{SS3.1} 
First we recall
that $\cB= \frac{1}{2}\br|\bfu|^2+\frac{1}{2}\br|\bfb|^2+p$. As
in subsection \ref{SS2.4}, we deduce that there exists a set
$\cT\subset(0,t_0)$ of the $1$-dimensional Lebesgue measure zero and $R_0>0$
such that for all $t\in(0,t_0)\smallsetminus\cT$,
\[
\sup_{R\in(0,R_0)} \int_{B_R(\bfx_0)}\frac{\cB_+(\bfx,t)}{|\bfx-\bfx_0|}\ \leq 1.
\]
Let $t\in(0,t_0)\smallsetminus\cT$, $\bfx_0\in\R^3$ and $R_0>0$. We
will use the identities (\ref{2.8}) in the form
\begin{align}
&\int_{B_R(\bfx_0)}\frac{1}{|\bfx-\bfx_0|}\ \bigl(2p+|\bfup|^2+
|\bfbr|^2\bigr)\, \rmd\bfy \nonumber \\
&=\ -\frac{1}{2R}\int_{B_R(\bfx_0)} \bigl(
|\bfu|^2+2\br|\bfb|^2\bigr)\, \rmd\bfx+\frac{3}{2R}
\int_{B_R(\bfx_0)} \Bigl(2p+|\bfu|^2+|\bfb|^2 \Bigr)\,
\rmd\bfy \nonumber \\
&=\ \int_{\R^3\smallsetminus B_R(\bfx_0)}
\frac{R^2}{|\bfx-\bfx_0|^3}\, \bigl[\br 2\br|\bfur|^2-|\bfup|^2-
2\br|\bfbr|^2+|\bfbp|^2 \br\bigr]\, \rmd\bfx.
\label{3.1}
\end{align}
Then, for $0<R\leq R_0$ and each time $t\in(0,t_0)$, we have
\begin{align}
&\frac{1}{2R} \int_{B_R(\bfx_0)}\bigl( |\bfu|^2+2\br
|\bfb|^2\bigr)\; \rmd\bfx\ \nonumber \\
&=\ \frac{3}{2R}\int_{B_R(\bfx_0)}\Bigl(
|\bfu|^2+|\bfb|^2+2p\Bigr)\; \rmd\bfx
-\int_{B_R(\bfx_0)} \frac{1}{|\bfx-\bfx_0|}\
\bigl( |\bfup|^2+ |\bfbr|^2+2p\bigr)\, \rmd\bfx \nonumber \\
&\leq\ \frac{3}{R}\int_{B_R(\bfx_0)} \cB_+\; \rmd\bfx+
\!\int_{B_R(\bfx_0)} \frac{1}{|\bfx-\bfx_0|}\ \bigl[ 2\cB_+-\bigl(
|\bfup|^2+ \bfbr|^2+2p\bigr)\bigr]\; \rmd\bfx
-\int_{B_R(\bfx_0)}
\frac{2\cB_+}{|\bfx-\bfx_0|}\; \rmd\bfx \nonumber \\
&\leq\ \int_{B_R(\bfx_0)}\frac{\cB_+}{|\bfx-\bfx_0|}\;
\rmd\bfx+\int_{B_R(\bfx_0)} \frac{1}{|\bfx-\bfx_0|}\ \bigl[
2\cB_+-\bigl( |\bfup|^2+|\bfbr|^2+2p\bigr)\bigr]\; \rmd\bfx, \nonumber \\
&=\ 3\int_{B_{R_0}(\bfx_0)}\frac{2\cB_+}{|\bfx-\bfx_0|}\;
\rmd\bfx 
-\int_{\R^3\smallsetminus B_{R_0}(\bfx_0)}
\frac{R_0^2}{|\bfx-\bfx_0|^3}\, \bigl[\br 2\br|\bfur|^2-|\bfup|^2-
2\br|\bfbr|^2+|\bfbp|^2 \br\bigr]\, \rmd\bfx. \label{3.2}
\end{align}
Obviously,
\begin{displaymath}
\biggl| \int_{\R^3\smallsetminus B_{R_0}(\bfx_0)}
\frac{R_0^2}{|\bfx-\bfx_0|^3}\, \bigl[\br 2\br|\bfur|^2-|\bfup|^2-
2\br|\bfbr|^2+|\bfbp|^2 \br\bigr]\, \rmd\bfx \biggr|\ \leq\
\frac{c}{R_0}\br,
\end{displaymath}
where the constant $c$ is independent of $\bfx_0$, $t$, $R$ and $R_0$. The
boundedness of the first term on the right hand side of
(\ref{3.2}), independent of $t$ for
$t\in(0,t_0)\smallsetminus\cT$, can now be justified by means
of the same arguments as the boundedness of the analogous
integral in subsection \ref{SS2.5}. The validity of the
premises in the implications (\ref{2.10}) and (\ref{2.11}) can
now be also confirmed in the same way as at the end of
subsection \ref{SS2.4}. The statements of these implications
and (\ref{2.3}) imply that (\ref{2.9}) holds.

The contents of subsections \ref{SS2.6} and \ref{SS2.7} can be
copied with the only change that we replace $p_-$ by $\cB_+$
and we also use the H\"older--continuity of $\bfu$ and $\bfb$
in the neighborhood of regular points. Instead of subsection
\ref{SS2.8}, where the proof of Theorem \ref{T1} was completed
under condition a), now we have the following subsection
\ref{SS3.2}.

\subsection{Completion of the proof of Theorem \ref{T1} under
condition b).} \label{SS3.2} 
Assume the condition b) of Theorem \ref{T1} holds. 
Let $t_0$ be an epoch of irregularity of the
solution $\bfu$, $\bfb$, $p$. Let $\bfx_0\in\R^3$. We will show
that there exists $R_*>0$ such that (\ref{2.16}) holds.

Let $R>0$ and $t\in(t_0-\delta,t_0]$. Using the identity between
the first two lines in (\ref{3.1}), and at the end also the
identity between the first and the third lines, we get
\begin{align}
& \frac{1}{2R} \int_{B_R(\bfx_0)}\bigl(|\bfu|^2+2\br|\bfb|^2\bigr)\; \rmd\bfx\  \nonumber \\
&=\ \frac{3}{2R}\int_{B_R(\bfx_0)}\bigl( |\bfu|^2+ |\bfb|^2+2p\bigr)\;
\rmd\bfx 
-\int_{B_R(\bfx_0)} \frac{1}{|\bfx-\bfx_0|}\
\bigl( |\bfup|^2+|\bfbr|^2+2p\bigr)\, \rmd\bfx \nonumber \\
\noalign{\vskip 2pt}
&\leq\ \frac{3}{R}\int_{B_R(\bfx_0)}\cB_+\; \rmd\bfx
-\int_{B_R(\bfx_0)}\frac{2\cB_+}{|\bfx-\bfx_0|}\; \rmd\bfx
\nonumber \\
&\quad +\int_{B_R(\bfx_0)} \frac{1}{|\bfx-\bfx_0|}\
\bigl[ 2\cB_+-\bigl( |\bfup|^2+|\bfbr|^2+2p\bigr)\bigr]\; \rmd\bfx
\nonumber \\ \noalign{\vskip 2pt}
&\leq\ \int_{B_R(\bfx_0)}\frac{\cB_+}{|\bfx-\bfx_0|}\; \rmd\bfx+
\int_{B_R(\bfx_0)} \frac{1}{|\bfx-\bfx_0|}\
\bigl(|\bfur|^2+|\bfbp|^2\bigr)\; \rmd\bfx \nonumber \\
\noalign{\vskip 2pt}
&\quad +\int_{B_R(\bfx_0)} \frac{1}{|\bfx-\bfx_0|}\
\bigl[ 2\cB_+-\bigl( |\bfu|^2+|\bfb|^2+ 2p\bigr)\bigr]\; \rmd\bfx,
\label{3.3}
\end{align}
which is finite, because it is the same as 
\begin{align}
&3\int_{B_R(\bfx_0)} \frac{\cB_+}{|\bfx-\bfx_0|}\ \;
\rmd\bfx+\int_{B_R(\bfx_0)} \frac{1}{|\bfx-\bfx_0|}\ \bigl(
|\bfur|^2+|\bfbp|^2 \bigr)\; \rmd\bfx \nonumber \\
\noalign{\vskip 2pt}
&\quad -\int_{B_R(\bfx_0)} \frac{1}{|\bfx-\bfx_0|}\
\bigl( |\bfu|^2+|\bfb|^2+2p\bigr)\; \rmd\bfx \nonumber \\
\noalign{\vskip 2pt}
&=\ 3\int_{B_R(\bfx_0)}\frac{\cB}{|\bfx-\bfx_0|}\;
\rmd\bfx-\int_{B_R(\bfx_0)} \frac{1}{|\bfx-\bfx_0|}\ \bigl(
|\bfup|^2+|\bfbr|^2+2p\bigr)\; \rmd\bfx \nonumber
\\ \noalign{\vskip 2pt}
&=\ 3\int_{B_R(\bfx_0)}\frac{\cB_+}{|\bfx-\bfx_0|}\;
\rmd\bfx \nonumber \\
&\quad -\int_{\R^3\smallsetminus B_R(\bfx_0)}
\frac{R^2}{|\bfx-\bfx_0|^3}\, \bigl[\br 2\br|\bfur|^2-|\bfup|^2-
2\br|\bfbr|^2+|\bfbp|^2 \br\bigr]\, \rmd\bfx \label{3.4} \\
\noalign{\vskip 2pt}
&\leq\ 3\int_{B_R(\bfx_0)}\frac{\cB_+}{|\bfx-\bfx_0|}\;
\rmd\bfx+\frac{c}{R}, \nonumber
\end{align}
where the constant $c$ is independent of $\bfx_0$, $t$ and $R$.

Let $\epsilon_*$ be the number in Lemma \ref{L2.2}. By analogy
with (\ref{2.18}), we have
\begin{equation}
\lim_{r\to 0+}\ \int_{B_r(\bfx_0)}
\frac{\cB_+(\bfx,t)}{|\bfx-\bfx_0|}\; \rmd\bfx\ =\ 0 \label{3.5}
\end{equation}
uniformly with respect to $t\in(0,t_0]$. Choose $R_*>0$ so small
that
\begin{equation}
3\int_{B_{R_*}(\bfx_0)}\frac{\cB_+(\bfx,t)}{|\bfx-\bfx_0|}\;
\rmd\bfx\ <\ \frac{\epsilon_*}{4} \label{3.6}
\end{equation}
for all $t\in(0,t_0]$ and
\begin{align}
&\int_{B_{R_*}(\bfx_0)} \frac{1}{|\bfx-\bfx_0|}\, \bigl(
|\bfur(\bfx,t_0)|^2+|\bfbp(\bfx,t)|^2\bigr)\; \rmd\bfx \nonumber \\
& +\int_{B_{R_*}(\bfx_0)} \frac{1}{|\bfx-\bfx_0|}\, \bigl[
\bigl(|\bfu(\bfx,t_0)|^2+|\bfb(\bfx,t_0)|^2+
2p(\bfx,t_0)\bigr)_{-}\bigr]\Bigr)\; \rmd\bfx\ <\
\frac{\epsilon_*}{4}. \label{3.7}
\end{align}
This choice of $\epsilon_*$ is possible, because of (\ref{3.5})
and due to the fact that the second integral in (\ref{3.7}) equals
the right hand side of (\ref{3.3}) (with $R=R_*$ and $t=t_0$),
which is finite and consists of three integrals over
$B_{R_*}(\bfx_0)$ with nonnegative integrands. Then, by analogy
with (\ref{2.21}) and from the comparison of (\ref{3.3}) with
(\ref{3.4}), we obtain
\begin{align*}
&-\int_{\R^3\smallsetminus B_R(\bfx_0)} \frac{R^2}
{|\bfx-\bfx_0|^3}\, \bigl[\br 2\br|\bfur|^2-|\bfup|^2-
2\br|\bfbr|^2+|\bfbp|^2 \br\bigr]\, \rmd\bfx\, \biggl|_{t=t_0} \\
&=\ -\int_{B_R(\bfx_0)}\frac{\cB_+}{|\bfx-\bfx_0|}\; \rmd\bfx+
\int_{B_R(\bfx_0)} \frac{1}{|\bfx-\bfx_0|}\
\bigl(|\bfur|^2+|\bfbp|^2\bigr)\; \rmd\bfx\, \biggl|_{t=t_0}
\nonumber \\ \noalign{\vskip 2pt}
&\quad +\int_{B_R(\bfx_0)} \frac{1}{|\bfx-\bfx_0|}\
\bigl(|\bfu|^2+|\bfb|^2+2p\bigr)_{-}\; \rmd\bfx\,
\biggl|_{t=t_0}\ \leq\ \frac{\epsilon_*}{4}\br.
\end{align*}
Due to (\ref{2.9}), there exists $\delta>0$ so small that
\begin{displaymath}
-\int_{\R^3\smallsetminus B_R(\bfx_0)} \frac{R^2}
{|\bfx-\bfx_0|^3}\, \bigl[\br 2\br|\bfur|^2-|\bfup|^2-
2\br|\bfbr|^2+|\bfbp|^2 \br\bigr]\, \rmd\bfx\ <\
\frac{\epsilon_*}{2}
\end{displaymath}
at all times $t\in(t_0-\delta,t_0]$. Applying this inequality with
(\ref{3.6}) and using the fact that the term on the left
hand side of (\ref{3.3}), which is $(2R)^{-1}\int_{B_R(\bfx_0)}
\bigl(|\bfu|^2+2\br|\bfb|^2\bigr)\; \rmd\bfx$, is equal to the
expression in (\ref{3.5}), we observe that the inequality
\begin{displaymath}
\frac{1}{2R}\int_{B_R(\bfx_0)} \bigl(|\bfu|^2+2\br|\bfb|^2\bigr)\;
\rmd\bfx\ <\ \frac{\epsilon_*}{4}+\frac{\epsilon_*}{2}\ =\
\frac{3\epsilon_*}{4}.
\end{displaymath}
holds for all $t\in(t_0-\delta,t_0]$ and $R\in(0,R_*]$. The proof
can now be completed in the same way as in the case of condition
a) in Section \ref{S2}.


\section*{Appendix}

Here, we return to the validity of formulas (\ref{2.5}) and
(\ref{2.6}). Recall that $\bfx\in\R^3$, $\bfx_0\in\R^3$, $R>0$
and $\alpha\in[0,1]$. We have
\begin{equation}
\int_{B_R(\bfx_0)}\frac{|\bfx_0-\bfy|^{-\alpha}}{|\bfx-\bfy|}\;
\rmd\bfy\ =\ \int_0^Rr^{-\alpha}\, \biggl(\int_{S_r(\bfx_0)}
\frac{\rmd_{\bfy}S}{|\bfx-\bfy|}\biggr)\; \rmd r. \tag{A1}
\end{equation}
The inside integral over $S_r(\bfx_0)$ depends on $\bfx$ only
through $|\bfx_0-\bfx|$. Thus, we may assume, without loss of
generality, that $\bfx_0=\bfzero$ and $\bfx=(0,0,ra)$, where
$a=|\bfx|/r$. We use the transformation to the spherical
coordinates: $\bfy\ =\ \bigl( r\, \cos\varphi\, \sin\vartheta,\,
r\, \sin\varphi\, \sin\vartheta, r\, \cos\vartheta\bigr)$. 
The Jacobian is equal to $r^2\, \sin\vartheta$. Then
\begin{displaymath}
|\bfx-\bfy|^2\ =\ \bigl|\, \bigl( r\, \cos\varphi\,
\sin\vartheta,\, r\, \sin\varphi\, \sin\vartheta,\, r\,
\cos\vartheta-ra \bigr) \bigr|^2\ =\ r^2\, \bigl[\br 1+a^2-2a\,
\cos\vartheta\br\bigr],
\end{displaymath}
and therefore, using also the change of variables $1+a^2-2a\,
\cos\vartheta=z$, we obtain
\begin{align*}
\int_{S_r(\bfzero)}\frac{\rmd_{\bfy}S}{|\bfx-\bfy|}\ &=\
2\pi\int_0^{\pi}\frac{r^2\, \sin\vartheta}{r\, \sqrt{1+a^2-2a\,
\cos\vartheta}}\; \rmd\vartheta\ =\ \frac{\pi r^2}{|\bfx|}
\int_{1+a^2-2a}^{1+a^2+2a} \frac{\rmd z}{\sqrt{z}} \\
&=\ \frac{2\pi r^2}{|\bfx|}\, \bigl[ \sqrt{z}
\bigr]_{(1-a)^2}^{(1+a)^2}\ =\ \frac{2\pi r^2}{|\bfx|}\, \bigl[
(1+a)\mp(1-a)],
\end{align*}
where the sign ``$-$'' holds if $1-a\geq 0$, which means
$|\bfx|\leq r$ and ``$+$'' holds if $1-a<0$, which means
$|\bfx|>r$. Thus, returning to a general point $\bfx_0$ instead of
the special case $\bfx_0=\bfzero$, we get
\begin{equation}
\int_{S_r(\bfx_0)}\frac{\rmd_{\bfy}S}{|\bfx-\bfy|}\ =\ \left\{
\begin{array}{ll} 4\pi r & \mbox{if}\ |\bfx-\bfx_0|\leq r, \\
4\pi r^2\, |\bfx-\bfx_0|^{-1} & \mbox{if}\ |\bfx-\bfx_0|>r.
\end{array} \right. \tag{A2}
\end{equation}

Let us at first assume that $|\bfx-\bfx_0|<R$. Then, by (A1),
\begin{align*}
\int_{B_R(\bfx_0)} \frac{|\bfx_0-\bfy|}{|\bfx-\bfy|}\; \rmd\bfy\
&=\ \int_0^Rr^{-\alpha}\int_{S_r(\bfx_0)}\frac{\rmd_{\bfy}S}
{|\bfx-\bfy|}\; \rmd r \\ \noalign{\vskip 1pt}
&=\ \int_0^{|\bfx_0-\bfx|} \frac{4\pi r^{2-\alpha}}
{|\bfx-\bfx_0|}\; \rmd r+\int_{|\bfx-\bfx_0|}^R4\pi r^{1-\alpha}\,
\rmd r \\ \noalign{\vskip 1pt}
&=\ \frac{4\pi\, |\bfx-\bfx_0|^{2-\alpha}}{3-\alpha}+\frac{4\pi
R^{2-\alpha}}{2-\alpha}-\frac{4\pi\, |\bfx-\bfx_0|^{2-\alpha}}
{2-\alpha}\ \\
&=\ \frac{4\pi\, R^{2-\alpha}}{2-\alpha}-\frac{4\pi\,
|\bfx-\bfx_0|^{2-\alpha}}{(3-\alpha)(2-\alpha)},
\end{align*}
which implies 
\begin{equation}
\nabla_{\bfx}^2 \int_{B_R(\bfx_0)} \frac{|\bfx_0-\bfy|}
{|\bfx-\bfy|}\; \rmd\bfy\ =\ -\frac{4\pi}{(3-\alpha)(2-\alpha)}\,
\nabla_{\bfx}^2\br |\bfx-\bfx_0|^{2-\alpha}. \tag{A3}
\end{equation}
Since
\begin{equation}
\nabla_{\bfx}^2\br |\bfx-\bfx_0|^{\beta}\ =\ \beta(\beta-2)\,
|\bfx-\bfx_0|^{\beta-4}\, (\bfx-\bfx_0)\otimes(\bfx-\bfx_0)+
\beta\, |\bfx-\bfx_0|^{\beta-2}, \tag{A4}
\end{equation}
equality (A3) (where we use (A4) with $\beta=\alpha-2$) yields
(\ref{2.5}).

Suppose now that $|\bfx-\bfx_0|>R$. Then (A1) and (A2) imply that
\[
\int_{B_R(\bfx_0)} \frac{|\bfx_0-\bfy|}{|\bfx-\bfy|}\; \rmd\bfy\
=\ \int_0^Rr^{-\alpha}\int_{S_r(\bfx_0)}\frac{\rmd_{\bfy}S}
{|\bfx-\bfy|}\; \rmd r\ =\ \int_0^R \frac{4\pi
r^{2-\alpha}}{|\bfx-\bfx_0|}\; \rmd r 
=\ \frac{4\pi R^{3-\alpha}}{3-\alpha}\, \frac{1}{|\bfx-\bfx_0|}.
\]
This together with (A4) (which we use with $\beta=-1$) 
yields (\ref{2.6}).

\section*{Acknowledgement}

J. Neustupa has been supported by the
Academy of Sciences of the Czech Republic (RVO 67985840) and by
the Grant Agency of the Czech Republic, grant No.~GA19-042435. 
M. Yang has been supported by the National Research
Foundation of Korea No. 2016R1C1B2015731 and No.~2015R1A5A1009350.

\end{document}